\documentclass{lms}

\usepackage{amsmath}
\usepackage{amssymb}
\usepackage{mathrsfs}
\usepackage{dsfont}

\usepackage{tikz}
\usetikzlibrary{matrix, arrows}

\usepackage{enumerate}

\newtheorem{thm}{Theorem}[section]
\newtheorem{intro-thm}{Theorem}[section]
\newtheorem{prop}[thm]{Proposition}
\newtheorem{cor}[thm]{Corollary}
\newtheorem{intro-cor}[intro-thm]{Corollary}
\newtheorem{lem}[thm]{Lemma}

\newnumbered{defi}[thm]{Definition}
\newnumbered{ex}[thm]{Example}
\newnumbered{rem}[thm]{Remark}

\providecommand{\ZZ}{\mathds Z}
\providecommand{\NN}{\ZZ_{\geq 0}}

\providecommand{\field}{K}
\providecommand{\hhh}{[| h |]}
\providecommand{\PSR}{\field \hhh}
\providecommand{\hzero}[1]{#1_{h=0}}

\providecommand{\alie}{\mathfrak a}
\providecommand{\blie}{\mathfrak b}
\providecommand{\glie}{\mathfrak g}
\providecommand{\hlie}{\mathfrak h}
\providecommand{\slt}{\mathfrak{sl}_2}

\providecommand{\End}{\mathrm{End}}

\providecommand{\Rep}{\mathrm{Rep}}
\providecommand{\rep}{\mathrm{rep}}

\providecommand{\Ncol}{N}
\providecommand{\qcol}{N_q}

\providecommand{\U}{U(\slt)}
\providecommand{\Uh}[1][\psi]{U_h(#1)}
\providecommand{\Uq}{U_h(\slt)}

\providecommand{\Ua}{U(\alie)}
\providecommand{\Uha}{U_h(\alie)}

\providecommand{\Verma}[1][n]{M(#1)}

\providecommand{\Vermah}[1][n,\psi]{M_h(#1)}

\providecommand{\Simple}[1][n]{L(#1)}

\providecommand{\Coeff}[1][d]{\mathbb X^{#1}}
\providecommand{\SCoeff}[1][d]{\mathbb X_\infty^{#1}}
\providecommand{\VCoeff}[1][d]{\mathbb X_V^{#1}}

\title[Coloured Kac-Moody algebras, Part I]{Coloured Kac-Moody algebras, Part I}
\author{Alexandre Bouayad}
\classno{16S30, 17B37 (primary), 16S80, 16T20, 17B10 (secondary)}

\begin{document}
\maketitle

\begin{abstract}
We introduce a parametrization of formal deformations of Verma modules of $\slt$. A point in the moduli space is called a colouring. We prove that for each colouring $\psi$ satisfying a regularity condition, there is a formal deformation $\Uh$ of $\U$ acting on the deformed Verma modules. We retrieve in particular the quantum algebra $\Uq$ from a colouring by $q$-numbers. More generally, we establish that regular colourings parametrize a broad family of formal deformations of the Chevalley-Serre presentation of $\U$. The present paper is the first of a series aimed to lay the foundations of a new approach to deformations of Kac-Moody algebras and of their representations. We will employ in a forthcoming paper coloured Kac-Moody algebras to give a positive answer to E. Frenkel and D. Hernandez's conjectures on Langlands duality in quantum group theory.
\end{abstract}

% ---------- Head of the document ---------- %

% ---------- Body of the document ---------- %

\section{Introduction}
\label{Intro}
\subsection{Deformation by Tannaka duality}
The Lie algebra $\slt$ formed by 2-by-2 matrices with zero trace is the easiest example of a semisimple Lie algebra, or more generally of a Kac-Moody algebra. The Chevalley-Serre presentation \cite{Serre} of $\slt$ consists of the Chevalley generators $H,X^-,X^+$, and of the relations
\begin{subequations}
\label{eq:Intro/Tannaka/slt-relations}
	\begin{align}
		\label{eq:Intro/Tannaka/slt-non-deformed-relations}
		[H, X^\pm] \, &= \, \pm 2 X^\pm, \\
		[X^-, X^+] \, &= \, H .
	\end{align}
\end{subequations}\par
We present in this paper a new approach, both elementary and systematic, to deformations of the universal enveloping algebra $\U$, over a ground field $\field$ of characteristic zero. Deformations here are formal, i.e.\ they are considered over the power series ring $\PSR$. We shall give a precision. It follows from a cohomological rigidity criterion of M. Gerstenhaber \cite{Gerstenhaber} that formal deformations of the structure of associative algebra of $\U$ are all trivial, i.e.\ they are conjugate to the constant formal deformation. In this paper though, we are interested in deforming a slightly richer structure, which consists of the algebra $\U$, together with the Chevalley generators. In other words, when considering a formal deformation of $\U$, we want to specify within it a deformation of the generators $H,X^-,X^+$. Equivalently, we may say that we are looking at formal deformations of the Chevalley-Serre presentation \eqref{eq:Intro/Tannaka/slt-relations} of $\U$.\par
Representations of $\slt$ carry all the information of the algebra $\U$, in the sense that $\U$ can be reconstructed by Tannaka duality from the category $\Rep(\slt)$ of representations of $\slt$. More specifically, $\U$ can be defined as the algebra of endomorphisms (namely, the natural transformations) of the forgetful functor from $\Rep(\slt)$ to the category of vector spaces.\par
We propose to construct formal deformations of $\U$ via Tannaka duality. In our view, the category $\Rep(\slt)$ would be too large to be deformed in one go. We need to look for a more modest subcategory to start with. One first candidate that comes easily in mind is the subcategory $\rep(\slt)$ of finite-dimensional representations. On the one hand, this subcategory is rich enough to distinguish by Tannaka duality elements in $\U$: an element in $\U$ is zero if and only if it acts by zero on every finite-dimensional representation of $\slt$. On the other hand, the category $\rep(\slt$) is notably elementary: all objects are completely reducible and finite-dimensional irreducible representations of $\slt$ are classified by their dimensions. There is however a slightly larger category which appears more suited to our purpose. This category is generated by Verma modules, in a sense which we will make precise in a next paper. Let us note that we consider here only integral Verma modules, that is to say Verma modules for which the action of $H$ has integral eigenvalues. One reason to prefer Verma modules rather than finite-dimensional representations is that the former are all equal when forgetting the action of $\slt$ -- they share the same underlying vector space. This makes deformation and Tannaka duality easier to deal with. Another reason is the definition itself of Verma modules of $\slt$ (namely, they are the representations of $\slt$ induced by the one-dimensional representations of a Borel subalgebra); we will show in a next paper that any deformation of Verma modules of $\slt$ lead naturally to a deformation of the whole category $\Rep(\slt)$.

\subsection{Summary of the main results}
\label{Intro/Results}
\begin{defi}
	A colouring is a sequence $\psi = (\psi^k(n))_{k \geq 1}$ of formal power series in $\PSR$, whose values depend on $n \in \ZZ$ and verify
	\begin{enumerate}
		\item[(C1)] $\psi^k(n) = k(n-k+1) \mod h$, for all $k,n$,
		\item[(C2)] $\psi^{n+1}(n) = 0$, for all $n \geq 0$,
		\item[(C3)] $\psi^{n+k+1}(n) = \psi^k(-n-2)$, for all $k \geq 1$ and all $n \geq 0$.
	\end{enumerate}
\end{defi}
For $n \in \ZZ$ we denote by $\Verma$ the integral Verma module of $\slt$ of highest weight $n$. When forgetting the action of $X^+$, the integral Verma modules of $\slt$ become representations of the Borel subalgebra $\blie$ spanned by $H$ and $X^-$. The action of $X^+$ can be retrieved from the \emph{natural colouring} $\Ncol$, defined by $\Ncol^k(n) = k(n-k+1)$. In view of axiom (C1), colourings can thus be considered as formal deformations of the action of $X^+$ on the integral Verma modules of $\slt$.
\begin{defi}
	We denote by $\Vermah$ the $\PSR$-module $\Verma \hhh$, endowed with the constant deformation of the action of $\blie$ on $\Verma$, together with the deformation of the action of $X^+$, given by the colouring $\psi$. 
\end{defi}
Here is where Tannaka duality comes into the picture.
\begin{defi}
	We denote by $\Uh$ the $\PSR$-algebra generated by $H,X^-,X^+$ and subject to the relations satisfied in every representation $\Vermah$.
\end{defi}
For the reader who may find unclear why this definition involves Tannaka duality, let us mention that there is a category built directly from the representations $\Vermah$ and whose Tannaka dual algebra is canonically isomorphic to $\Uh$. Details will appear in a next paper.\par
We prove that $\Uh$ deforms the algebra $\U$, provided that the colouring $\psi$ satisfies a regularity condition; we call it then a \emph{coloured Kac-Moody algebra}.
\begin{intro-thm}
	The algebra $\Uh$ is a formal deformation of the algebra $\U$ if and only if the colouring $\psi$ is regular, i.e.\ $\psi^k(n) \in \field[k,n] \hhh$.
\end{intro-thm}
A coloured Kac-Moody algebra defines not only a formal deformation of the algebra $\U$, but also a formal deformation of the Chevalley generators of $\U$. As a result, it defines unambiguously -- once we have fixed a basis of $\U$, e.g.\ the PBW basis formed by the monomials $(X^-)^a (X^+)^b H^c$ ($a,b,c \geq 0$) -- a formal deformation of the Chevalley-Serre presentation of $\U$.
\begin{intro-thm}
	For $\psi$ regular, the $\PSR$-algebra $\Uh$ is generated by $H,X^-,X^+$, and subject to the relations
	\begin{align*}
		[H, X^\pm] \, &= \, \pm 2 X^\pm, \\
		X^+X^- \, &= \, \sum_{a = 0}^\infty (X^-)^a (X^+)^a \xi^a(H),
	\end{align*}
	where the series $\xi^a(H) \in \field[H] \hhh$ form the regular solution of an infinite-dimensional linear equation parametrized by $\psi$ (see section \ref{Eqh}).
\end{intro-thm}
Let $\Ua$ be the $\field$-algebra generated by $H,X^-,X^+$ and subject to the relations \eqref{eq:Intro/Tannaka/slt-non-deformed-relations}. There is a canonical homomorphism from $\Ua$ to $\U$; we say that $\U$ is an $\alie$-algebra. Relations \eqref{eq:Intro/Tannaka/slt-non-deformed-relations} hold in $\Uh$ for every colouring $\psi$, i.e.\ $\Uh$ is a $\Uha$-algebra, where $\Uha$ designates the $\PSR$-algebra $\Ua \hhh$. We may then regard a coloured Kac-Moody algebra $\Uh$ as a formal deformation of the structure of $\alie$-algebra of $\U$.\par
To any symmetrizable Kac-Moody algebra $\glie$, V. Drinfel'd \cite{Drinfeld1} and M. Jimbo \cite{Jimbo} associated a formal deformation $U_h(\glie)$ of the universal enveloping algebra $U(\glie)$.\footnote{The quantum algebra $\Uq$ first appeared in works of P. Kulish and E. Sklyanin; see \cite{Kulish-Sklyanin}.} We show in this paper that $\Uq$ arises from the \emph{$q$-colouring} $\qcol$, which is defined by replacing natural numbers with $q$-numbers in the natural colouring $\Ncol$.
\begin{intro-thm}
	The quantum algebra $\Uq$ is isomorphic as a $\Uha$-algebra to the coloured Kac-Moody algebra $\Uh[\qcol]$.
\end{intro-thm} 
It has been proved by V. Drinfel'd \cite{Drinfeld2} that for $\glie$ semisimple, $U_h(\glie)$ is a $\hlie$-trivial formal deformation of $U(\glie)$, i.e.\ there exists an equivalence of formal deformation between $U(\glie) \hhh$ and $U_h(\glie)$ fixing the Cartan subalgebra $\hlie$ of $\glie$. We establish in the present paper that regular colourings classify all $\hlie$-trivial formal deformations of the structure of $\alie$-algebra of $\U$.
\begin{intro-thm}
	The map $\psi \mapsto \Uh$ is a bijection between colourings and isomorphism classes of $\hlie$-trivial deformations of the $\alie$-algebra $\U$.
\end{intro-thm}
As a corollary, we obtain a new\footnote{To the best of the author's knowledge.} rigidity result for $\U$.
\begin{intro-cor}
	A $\hlie$-trivial formal deformation $A$ of the $\alie$-algebra $\U$ is also $\blie$-trivial, i.e.\ there exists an equivalence of formal deformation between $\U \hhh$ and $A$ fixing both $H$ and $X^-$, and $A$ moreover admits only one such equivalence.
\end{intro-cor}

\subsection{Coloured Kac-Moody algebras}
The present paper is the first of a series aimed to lay the foundations of a new approach to deformations of Kac-Moody algebras, and of their representations. We present here results in the rank one case, focusing on one-parameter formal deformations of the $\alie$-algebra structure of $\U$. We will investigate in a next paper formal deformations of the structure of Hopf algebra of $\U$. It will be proved that regular (di)colourings provide a classification of formal deformations of the Hopf algebra $\U$ (together with formal deformations of the Chevalley generators). More generally, we will show that there is a natural group action on the set of colourings and that the resulting action groupoid is equivalent to the groupoid formed by formal deformations of the Hopf algebra $\U$. These results will be generalized in subsequent papers from $\slt$ to any symmetrizable Kac-Moody algebra $\glie$. Let us precise that we won't be concerned with all the deformations of the Hopf algebra $U(\glie)$, as we will restrict ourselves to those deformations which preserve the grading of $U(\glie)$ by the weight lattice of $\glie$.\par
Coloured Kac-Moody algebras are defined by Tannaka duality. In a next paper, we will explain how a colouring $\psi$ induces in an elementary way a closed monoidal category $\Rep (\glie,\psi)$, and we will show that $\Rep (\glie,\psi)$ is a deformation of the category of all representations of $\glie$. The coloured Kac-Moody algebra $U(\glie,\psi)$ will be defined as the Hopf algebra corresponding by Tannaka duality to the category $\Rep (\glie,\psi)$. Let us note that whereas the construction of the category $\Rep (\glie,\psi)$ is aimed to be as elementary as possible, the coloured Kac-Moody algebra $U(\glie,\psi)$ itself may be in general difficult to describe explicitly (consider for example the Chevalley-Serre presentation of $\Uh$; see theorem \ref{thm:CKM/Generators-and-relations}).\par
Coloured Kac-Moody algebras should be understood as multi-parameters deformations of usual Kac-Moody algebras, with as many deformation parameters as there are degrees of freedom in the choice of a colouring. We will show in a next paper that all constructions and results obtained over the power series ring $\PSR$ hold over more general rings. There is in fact a generic coloured Kac-Moody algebra, which is universal in the sense that every other coloured Kac-Moody algebra can be obtained from it by specializing generic deformation parameters. Specializations will be a key feature of coloured Kac-Moody algebras, with several applications, as for example crystallographic Kac-Moody algebras to name one (we will show that there is a natural correspondence between representations of crystallographic Kac-Moody algebras and crystals of representations of quantum Kac-Moody algebras).

\subsection{Langlands interpolation}
P. Littelmann \cite{Littelmann} and K. McGerty \cite{McGerty} have revealed the existence of relations between representations of a symmetrizable Kac-Moody algebra $\glie$ and representations of its Langlands dual ${}^L \glie$ (the Kac-Moody algebra defined by transposing the Cartan matrix of $\glie$). They have proved that the action of the quantum algebra $U_q(\glie)$ on certain representations interpolates between an action of $\glie$ and an action of ${}^L \glie$. Namely, they have shown that the actions of $\glie$ and ${}^L \glie$ can be retrieved from the action of $U_q(\glie)$ by specializing the parameter $q$ to $1$, and to some root of unity $\epsilon$, respectively. In the case where $\glie$ is a finite-dimensional simple Lie algebra, E. Frenkel and D. Hernandez introduced in \cite{Frenkel-Hernandez} an algebra depending on an additional parameter $t$, and they conjectured the existence of representations for this algebra interpolating between representations of the quantum algebras $U_q(\glie)$ and $U_t({}^L \glie)$. They besides conjectured that the constructions could be extended to any symmetrizable Kac-Moody algebra $\glie$. They lastly suggested a Langlands duality for crystals.\par
We will give in a forthcoming paper a positive answer to these conjectures. More precisely, we will show that for any symmetrizable Kac-Moody algebra $\glie$ there exists a coloured Kac-Moody algebra $U(\glie,\Ncol_{q,t})$ whose representations possess the predicted interpolation property. Examples of such an algebra have been explicitly constructed by the author of this paper in the case of $\slt$ (within the more general context of isogenies of root data); see \cite{Bouayad}. Using crystallographic Kac-Moody algebras, we will moreover confirm manifestations of Langlands duality at the level of crystals. Let us precise that the coloured Kac-Moody algebra $U(\glie,\Ncol_{q,t})$ have relations with E. Frenkel and D. Hernandez's algebra only when $\glie = \slt$. For $\glie$ of higher ranks, the two algebras differ significantly.\par
We lastly mention that Langlands duality for quantum groups might have promising connections with the geometric Langlands correspondence; see \cite{Frenkel-Hernandez} and \cite{Frenkel-Reshetikhin}.

\subsection{Organization of the paper}
\label{Intro/Organization}
In section \ref{Preliminaries}, we introduce colourings and we construct deformed Verma modules $\Vermah$ associated to a colouring $\psi$. We briefly recall the notion of formal deformation of associative algebras, and we give several definitions suited to the context of this paper. We define the algebra $\Uh$, and we prove that $\Uh$ is a formal deformation of an extension of $\U$. In section \ref{Eqh}, we express the action of $\Uh$ on $\Vermah$ in terms of infinite-dimensional linear equations (proposition \ref{prop:Eqh/Interpretation}). We prove that these equations always admit regular solutions if and only if $\psi$ is regular (proposition \ref{prop:Eqh/Regular-solutions}); this is the key technical result of this paper. In section \ref{CKM}, we prove that $\Uh$ is a formal deformation of $\U$ if and only if the colouring $\psi$ is regular (theorem \ref{thm:CKM/Formal-deformations}). We give a Chevalley-Serre presentation of the coloured Kac-Moody algebra $\Uh$ (theorem \ref{thm:CKM/Generators-and-relations}). We show that the constant formal deformation $\U \hhh$ and the quantum algebra $\Uq$ can be realized as coloured Kac-Moody algebras (theorem \ref{thm:CKM/Realizations}). We prove that coloured Kac-Moody algebras are $\blie$-trivial deformations of $\U$, and admit unique $\blie$-trivializations (theorem \ref{thm:CKM/b-triviality}). We prove that regular colourings classify $\hlie$-trivial formal deformations of the $\alie$-algebra $\U$ (theorem \ref{thm:CKM/Classification}). As a corollary, we obtain a rigidity result for $\U$; namely, every $\hlie$-trivial formal deformation of the $\alie$-algebra $\U$ is also $\blie$-trivial, and admits a unique $\blie$-trivialization (corollary \ref{cor:CKM/b-triviality}).

\section{Preliminaries}
\label{Preliminaries}
In this section, we introduce colourings, and we construct deformed Verma modules $\Vermah$ associated to a colouring $\psi$. We briefly recall the notion of formal deformations of associative algebras, and we give a few definitions suited to the context of this paper. We define the algebra $\Uh$, and we prove that $\Uh$ is a formal deformation of an extension of the algebra $\U$.

\subsection*{Notations and conventions}
\subsubsection*{(i)} The integers are elements of $\ZZ = \{ \dotsc, -2, -1, 0, 1, 2, \dotsc \}$. The non-negative integers are elements of $\NN = \{ 0, 1, 2, \dotsc \}$. An empty sum is equal to zero, and an empty product is equal to one. We recall that $\field$ designates a field of characteristic zero. We denote by $\field^\ZZ$ the $\field$-vector space formed by functions from $\ZZ$ to $\field$.

\subsubsection*{(ii)} We denote by $\PSR$ the power series ring in the variable $h$ over the field $\field$. An element $\lambda$ in $\PSR$ is of the form $\lambda = \sum_{m \geq 0} \lambda_m h^m$ with $\lambda_m \in \field$. We regard $\field$ as a subring of $\PSR$.

\subsubsection*{(ii)} Associative algebras and their homomorphisms are unital. Representations are left. Let $\glie$ be a Lie algebra over $\field$, its universal enveloping algebra is denoted by $U(\glie)$. We identify as usual representations of $\glie$ with representations of $U(\glie)$.

\subsubsection*{(iv)} Let $B$ be a $R$-algebra ($R = \field, \PSR$). A $B$-algebra is a $R$-algebra $A$, together with a structural $R$-algebra homomorphism $f : B \to A$. Let $A'$ be another $B$-algebra, a $B$-algebra homomorphism from $A$ to $A'$ is a $R$-algebra homomorphism $g : A \to A'$ such that $g \circ f = f'$, where $f'$ designates the structural homomorphism from $B$ to $A'$. By $\glie$-algebra (for $\glie$ a Lie algebra over $\field$) we understand $U(\glie)$-algebra.

\subsubsection*{(v)} For $V_0$ a $\field$-vector space, we denote by $V_0 \hhh$ the $\PSR$-module formed by series of the form $v = \sum_{m \geq 0} v_m h^m$ with $v_m \in V_0$. We regard $V_0$ as a $\field$-vector subspace of $V_0 \hhh$. A structure of $\field$-algebra on $V_0$ induces a structure of $\PSR$-algebra on $V_0 \hhh$. Similarly, a representation $V_0$ of a $\field$-algebra $B_0$ yields a representation $V_0 \hhh$ of the $\PSR$-algebra $B_0 \hhh$. More generally, a $B_0$-algebra $A_0$ yields a $B_0 \hhh$-algebra $A_0 \hhh$.

\subsubsection*{(vi)} For $V$ a $\PSR$-module, we denote by $\hzero V$ the $\field$-vector space $V / h V$. For $f : V \to W$ a $\PSR$-linear map, we denote by $\hzero f$ the induced $\field$-linear map from $\hzero V$ to $\hzero W$. A structure of $\PSR$-algebra on $V$ induces a structure of $\field$-algebra on $\hzero V$. Similarly, a representation $V$ of a $\PSR$-algebra $B$ yields a representation $\hzero V$ of the $\field$-algebra $\hzero B$, and a $B$-algebra $A$ yields a $(\hzero B)$-algebra $\hzero A$.

\subsubsection*{(vii)} The $h$-adic topology of a $\PSR$-module $V$ is the linear topology whose local base at zero is formed by the $\PSR$-submodules $h^m V$ ($m \in \NN$). A $\PSR$-module $V$ is said topologically free if $V$ is isomorphic to $V_0 \hhh$ for some $\field$-vector space $V_0$, or equivalently, if $V$ is Hausdorff, complete and torsion-free; see for example \cite{Bourbaki}. A $\PSR$-algebra $A$ is said topologically generated by a subset $A'$, if $A$ is equal to the closure of the $\PSR$-subalgebra generated by $A'$.

\subsubsection*{(viii)}  We will often make use of the following two facts for a $\PSR$-linear map $f : V \to W$:\par
- $\hzero f$ surjective implies $f$ surjective, if $V$ is complete, and if $W$ is Hausdorff,\par
- $\hzero f$ injective implies $f$ injective, if $V$ is Hausdorff, and if $W$ is torsion-free.

\subsubsection*{(ix)} The Chevalley generators $X^-,X^+,H$ form a basis of $\slt$. It then follows from the Poincar\'e-Birkhoff-Witt theorem that the monomials $(X^-)^a (X^+)^b H^c$ ($a,b,c \geq 0$) form a basis of $\U$; we call it the PBW basis.

\subsection{Colourings}
We call the relations
\begin{equation}
\label{eq:Preliminaries/Colourings/slt-non-deformed-relations}
\tag{\ref{eq:Intro/Tannaka/slt-non-deformed-relations}}
	[H, X^\pm] \, = \, \pm 2 X^\pm
\end{equation}
the \emph{non-deformed relations} of $\slt$. We denote by $\alie$ the Lie algebra over $\field$ generated by $H,X^-,X^+$, and subject to the non-deformed relations of $\slt$. We are interested in this paper in deformations of the algebra $\U$, where the non-deformed relations of $\slt$ still hold. Representations of such deformations of $\U$ are in particular representations of the constant formal deformation $\Ua \hhh$, which we will denote by $\Uha$.\par
We denote by $\blie^+$ the Borel subalgebra of $\slt$ spanned by $H$ and $X^+$. We recall that a Verma module of $\slt$ is a representation of $\slt$ induced from a one-dimensional representation of $\blie^+$. For $n \in \ZZ$ the (integral) Verma module $\Verma$ of highest weight $n$ is equal to $\bigoplus_{k \geq 0} \field b_k$ as a vector space, and the action of $\slt$ is given by
\begin{equation}
\label{eq:Preliminaries/Colourings/Verma}
	\begin{aligned}
		H . b_k \, &= \, (n-2k) b_k , \\
		X^- . b_k \, &= \, b_{k+1}, \\
		X^+ . b_k \, &= \,
		\begin{cases}
			0& \text{if $k = 0$,} \\
			k (n - k + 1) b_{k-1}& \text{if $k \geq 1$.}
		\end{cases}
	\end{aligned}
\end{equation}
We remark that $H,X^-,X^+$ act on $\Verma$ as scalar multiplications between $\NN$ copies of $\field$:
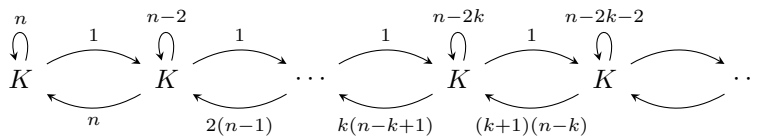
\begin{center}
	\begin{tikzpicture}[>=stealth, every loop/.style={<-}]
		\matrix (a) [matrix of math nodes, column sep=1.25cm] {
			\, \field \, & \, \field \, & \phantom{\, \field \,} & \, \field \, & \, \field \, & \phantom{\, \field \,} \\
		};
		\node at (a-1-3) {$\cdots$};
		\node at (a-1-6) {$\cdots$};
		\path[->]
			(a-1-1) edge [loop above] node[above]{$\scriptstyle{n}$} (a-1-1)
			(a-1-2) edge [loop above] node[above]{$\scriptstyle{n-2}$} (a-1-2)
			(a-1-4) edge [loop above] node[above]{$\scriptstyle{n-2k}$} (a-1-3)
			(a-1-5) edge [loop above] node[above]{$\scriptstyle{n-2k-2}$} (a-1-4)
			(a-1-1) edge [bend left] node[above]{$\scriptstyle{1}$} (a-1-2)
			(a-1-2) edge [bend left] node[below]{$\scriptstyle{n}$} (a-1-1)
			(a-1-2) edge [bend left] node[above]{$\scriptstyle{1}$} (a-1-3)
			(a-1-3) edge [bend left] node[below]{$\scriptstyle{2(n - 1)}$} (a-1-2)
			(a-1-3) edge [bend left] node[above]{$\scriptstyle{1}$} (a-1-4)
			(a-1-4) edge [bend left] node[below]{$\scriptstyle{k(n - k + 1)}$} (a-1-3)
			(a-1-4) edge [bend left] node[above]{$\scriptstyle{1}$} (a-1-5)
			(a-1-5) edge [bend left] node[below]{$\scriptstyle{(k+1)(n - k)}$} (a-1-4)
			(a-1-5) edge [bend left] (a-1-6)
			(a-1-6) edge [bend left] (a-1-5)
		;
	\end{tikzpicture}
\end{center}
Let $\mathcal B^+$ designate the quiver formed by the bottom arrows of the previous graph. We can think of the action of $X^+$ on the integral Verma modules of $\slt$ as a $\ZZ$-graded representation of the quiver $\mathcal B^+$. This representation, which we denote by $\Ncol$, assigns to each vertex of $\mathcal B^+$ the $\field$-vector space $\bigoplus_{n \in \ZZ} \field$, and assigns to the $k$-th arrow ($k \geq 1$) the $\field$-linear map represented by the diagonal matrix $\mathrm{diag}(\Ncol^k(n); n \in \ZZ)$ where $\Ncol^k(n) = k(n-k+1)$.\par
Fixing the actions of $H$ and $X^-$, a formal deformation of the action of the Lie algebra $\alie$ on the integral Verma modules of $\slt$ corresponds to a formal deformation of the $\ZZ$-graded representation $\Ncol$ of the quiver $\mathcal B^+$. Such a deformation is specified by a deformation $\psi^k(n)$ in $\PSR$ of the scalar $\Ncol^k(n)$, for each $k \geq 1$ and for each $n \in \ZZ$:
\begin{center}
	\begin{tikzpicture}[>=stealth]
		\matrix (a) [matrix of math nodes, column sep=1.25cm] {
			\, \bullet \, & \, \bullet \, & \, \cdots \, & \, \bullet \, & \, \bullet \, & \, \cdots \\
		};
		\path[->]
			(a-1-2) edge node[above]{$\scriptstyle{\psi^1(n)}$} (a-1-1)
			(a-1-3) edge node[above]{$\scriptstyle{\psi^2(n)}$} (a-1-2)
			(a-1-4) edge node[above]{$\scriptstyle{\psi^k(n)}$} (a-1-3)
			(a-1-5) edge node[above]{$\scriptstyle{\psi^{k+1}(n)}$} (a-1-4)
			(a-1-6) edge (a-1-5)
		;
	\end{tikzpicture}
\end{center}
We recall that for each $n \geq 0$ there is a non-zero morphism from $\Verma[-n-2]$ to $\Verma$. In terms of the representation $\Ncol$, the property becomes $\Ncol^{n+1}(n) = 0$ and $\Ncol^{n+k+1}(n) = \Ncol^k(-n-2)$ for all $k \geq 1$ and all $n \geq 0$. A formal deformation of the representation $\Ncol$ preserving these conditions is called a colouring.
\begin{defi}
	A colouring is a sequence $\psi = (\psi^k)_{k \geq 1}$ with values in $\field^\ZZ \hhh$, verifying
	\begin{enumerate}
		\item[(C1)] $\psi^k(n) = k(n-k+1) \mod h$, for all $k,n$,
		\item[(C2)] $\psi^{n+1}(n) = 0$, for all $n \geq 0$,
		\item[(C3)] $\psi^{n+k+1}(n) = \psi^k(-n-2)$, for all $k \geq 1$ and all $n \geq 0$.
	\end{enumerate}
\end{defi}
As explained before, a colouring is meant to encode a formal deformation of the action of $X^+$ on the integral Verma modules of $\slt$, in such a way that the non-deformed relations of $\slt$ remain satisfied.
\begin{defi}
	Let $\psi$ be a colouring, and let $n \in \ZZ$. We denote by $\Vermah$ the representation of $\Uha$, whose underlying $\PSR$-module is $(\bigoplus_{k \geq 0} \field b_k) \hhh$, and where the action of $\Uha$ is given by
	\begin{equation*}
		\begin{aligned}
			H . b_k \, &= \, (n-2k) b_k, \\
			X^- . b_k \, &= \, b_{k+1}, \\
			X^+ . b_k \, &= \,
			\begin{cases}
				0& \text{if $k = 0$,} \\
				\psi^k(n) b_{k-1} & \text{if $k \geq 1$.}
			\end{cases}
		\end{aligned}
	\end{equation*}
\end{defi}
\begin{ex}
	We call \emph{natural colouring}, and we denote by $\Ncol$, the unique colouring with values in $\field^\ZZ$; it is defined by $\Ncol^k(n) = k(n-k+1)$. The natural colouring encodes the action of $\alie$ on the integral Verma modules of $\slt$: $\Vermah[n,\Ncol] = \Verma \hhh$ as representations of $\Uha$ for all $n \in \ZZ$.
\end{ex}
\begin{ex}
\label{ex:Preliminaries/Colourings/q-colouring}
	The quantum algebra $\Uq$ is the $\Uha$-algebra topologically generated by $H,X^-,X^+$, and subject to the relation
	\begin{equation*}
		[X^+, X^-] \, = \, \frac{q^H  - q^{-H}}{q - q^{-1}} \quad \text{with $q = \exp(h)$ and $q^H = \exp(hH)$,}
	\end{equation*}
	i.e.\ $\Uq$ is the quotient of $\Uha$ by the smallest closed (for the $h$-adic topology) two-sided ideal containing the previous relation. We denote by $\qcol$ the colouring defined by
	\begin{equation*}
		\qcol^k(n) = [k]_q [n-k+1]_q \quad \text{where $[k]_q = \frac{q^k - q^{-k}}{q - q^{-1}}$;}
	\end{equation*}
	we call it the \emph{$q$-colouring}. The $q$-colouring encodes the action of $\Uha$ on the integral Verma modules of $\Uq$: for all $n \in \ZZ$, the representation $\Vermah[n,\qcol]$ is the Verma module of $\Uq$ of highest weight $n$, when viewed as a representation of $\Uha$.
\end{ex}

\subsection{Formal deformation of associative algebras}
By formal deformation of a $\field$-algebra $A_0$, one usually designates a topologically free $\PSR$-algebra $A$, together with a $\field$-algebra isomorphism $f_0$ from $A_0$ to $\hzero A$. Two formal deformations $(A,f_0)$ and $(A',f'_0)$ are said equivalent if there exists a $\PSR$-algebra isomorphism $g$ from $A$ to $A'$ such that $\hzero g \circ f_0 = f'_0$.\par
As already mentioned, we are interested in this paper in formal deformations of $\U$ where the non-deformed relations \eqref{eq:Intro/Tannaka/slt-non-deformed-relations} of $\slt$ still hold. In other words, the deformations of interest will be the formal deformations of $\U$ within the
category of $\alie$-algebras ($\U$ has a canonical structure of $\alie$-algebra, induced by the projection map from $\alie$ to $\slt$). Let us remark that specifying a $B$-algebra structure on an algebra $A$ not only forces every relation in the algebra $B$ to be satisfied in $A$, but also fixes pointwise in $A$ the image of $B$. A formal deformation of the $\alie$-algebra $\U$ should therefore be understood as a formal deformation of the $\field$-algebra $\U$, together with a formal deformation of the Chevalley generators $H,X^-,X^+$ within the deformed algebra, in such a way that the
non-deformed relations of $\slt$ are preserved.
\begin{defi}
\label{defi:Preliminaries/Formal-deformations/algebra}
	Let $B_0$ be a $\field$-algebra, and let $A_0$ be a $B_0$-algebra. We suppose that the structural homomorphism from $B_0$ to $A_0$ is surjective. A formal deformation of the $B_0$-algebra $A_0$ is a $B_0 \hhh$-algebra $A$ such that
	\begin{enumerate}
		\item[(D1)] the $\PSR$-module $A$ is topologically free,
		\item[(D2)] the $B_0$-algebras $\hzero A$ and $A_0$ are isomorphic.
	\end{enumerate}
\end{defi}
Axiom (D2) may need a precision: the structure of $B_0 \hhh$-algebra on $A$ induces a structure of $(\hzero{B_0 \hhh})$-algebra on $\hzero A$, and thus a structure of $B_0$-algebra (the $\field$-algebras $\hzero{B_0 \hhh}$ and $B_0$ are canonically isomorphic).\par
As the structural homomorphism from $B_0$ to $A_0$ is surjective, there is a unique way to identify the $B_0$-algebras $\hzero A$ and $A_0$. This shows that definition \ref{defi:Preliminaries/Formal-deformations/algebra} extends the usual definition of a formal deformation of a $\field$-algebra. Let us also remark that in view of axiom (D1) the structural homomorphism from $B_0 \hhh$ to $A$ is necessarily surjective.
\begin{ex}
	The quantum algebra $\Uq$, as defined in example \ref{ex:Preliminaries/Colourings/q-colouring}, is a formal deformation of the $\alie$-algebra $\U$.
\end{ex}
\begin{defi}
	Let $B_0$ be a $\field$-algebra and let $V_0$ be a representation of $B_0$. A formal deformation of $V_0$ along $B_0 \hhh$ is a representation $V$ of $B_0 \hhh$ such that
	\begin{enumerate}
		\item[(D1')] the $\PSR$-module $V$ is topologically free,
		\item[(D2')] the representations $\hzero V$ and $V_0$ are isomorphic.
	\end{enumerate}\par
	Let $A$ be a formal deformation of a $B_0$-algebra $A_0$. We suppose that the action of $B_0$ on $V_0$ factorises through $A_0$. We say that $V$ is a formal deformation of $V_0$ along $A$ if the action of $B_0 \hhh$ on $V$ factorises through $A$.
\end{defi}
\begin{ex}
	For every $n \in \ZZ$, the representation $\Vermah[n,\qcol]$, where $\qcol$ designates the $q$-colouring (see example \ref{ex:Preliminaries/Colourings/q-colouring}), is a formal deformation along $\Uq$ of $\Verma$.
\end{ex}
The following lemma gives further examples of formal deformations of representations along the algebra $\Uha$. It follows immediately from the first colouring axiom.
\begin{lem}
\label{lem:Preliminaries/Formal-deformations/Vermah}
	Let $\psi$ be a colouring. For every $n \in \ZZ$, the representation $\Vermah$ is a formal deformation along $\Uha$ of the Verma module $\Verma$, when viewed as a representation of $\Ua$.
\end{lem}

\subsection{The algebra $\Uh$}
Integral Verma modules of $\slt$ distinguish elements in the algebra $\U$, i.e.\ an element in $\U$ is zero if and only if it acts by zero on $\Verma$ for all $n \in \ZZ$. This remains true if we replace integral Verma modules of $\slt$ with finite-dimensional irreducible representations of $\slt$. These are known facts. We give a proof of them, for the reader's convenience.
\begin{prop}
\label{prop:Preliminaries/Uh/perfect}
	Let $x \in \U$. The following assertions are equivalent.
	\begin{enumerate}
		\item The element $x$ is zero.
		\item The element $x$ acts by zero on all the integral Verma modules of $\slt$.
		\item The element $x$ acts by zero on all the finite-dimensional irreducible representations of $\slt$.
	\end{enumerate}
\end{prop}
\begin{proof}
The implication (i) $\Rightarrow$ (ii) is immediate. For $n \geq 0$ we denote by $\Simple$ the unique (up to isomorphism) irreducible representation of $\slt$ of dimension $n+1$. The representation $\Simple$ is a quotient of the integral Verma module $\Verma$. As a consequence, assertion (ii) implies assertion (iii). Let us prove that assertion (iii) implies assertion (i). Let $x$ be a non-zero element in $\U$, and let us suppose for the sake of contradiction that $x$ acts by zero on $\Simple$ for all $n \in \ZZ$. The representation $\Simple$ is equal to $\bigoplus^n_{k = 0} \field b_k$ as a vector space, and the action of $\slt$ is given by
\begin{equation}
\label{eq:Preliminaries/Uh/Simple}
	\begin{aligned}
		H . b_k \, &= \, (n-2k) b_k , \\
		X^- . b_k \, &= \,
		\begin{cases}
			b_{k+1}& \text{if $k \leq n-1$,} \\
			0& \text{if $k = n$,}
		\end{cases} \\
		X^+ . b_k \, &= \,
		\begin{cases}
			0& \text{if $k = 0$,} \\
			k (n - k + 1) b_{k-1}& \text{if $k \geq 1$.}
		\end{cases}
	\end{aligned}
\end{equation}
The $\field$-algebra $\U$ has a $\ZZ$-gradation, defined by $\deg(H) = 0$ and $\deg(X^\pm) = \pm 1$. According to the way $\U$ acts on $\Simple$, we can assume without loss of generality that $x$ is a homogeneous element of $\U$. Let $d$ designate the degree of $x$. In view of the PBW basis of $\U$, $x = \sum_{a = a_1}^{a_2} (X^-)^{a-d} (X^+)^a \xi^a (H)$ for some $\xi^{a_1}(H), \dotsc, \xi^{a_2}(H) \in \field[H]$, with $a_2 \geq a_1 \geq \max(0,d)$. As $x$ is non-zero, we can assume that $\xi^{a_1}(H) \neq 0$. Let $n \geq a_1,a_1-d$. According to \eqref{eq:Preliminaries/Uh/Simple}, the action of $x$ on $\Simple$ satisfies
\begin{equation*}
	x . b_{a_1} \, = \, (X^-)^{a_1-d} (X^+)^{a_1} \xi^{a_1} (H) . b_{a_1} \, = \, \frac{a_1! \, n!}{(n-a_1)!} \xi^{a_1}(n-2a_1) b_{a_1-d}.
\end{equation*}
It follows that $\xi^{a_1}(n-2a_1)$ is zero for all $n \geq a_1,a_1-d$. The polynomial $\xi^{a_1}(H)$ is therefore zero, which is a contradiction.
\end{proof}
In other words, one can define the algebra $\U$ as the quotient of $\Ua$ by $I$, where $I$ designates the two-sided ideal of $\Ua$ formed by the elements acting by zero on all the integral Verma modules of $\slt$, when viewed as representations of $\Ua$. This construction of $\U$ may be viewed as an expression of a Tannaka duality between the algebra $\U$ on the one side, and the Verma modules $\Verma$ on the other. We propose to consider the same construction, where integral Verma modules of $\slt$ now carry a ``coloured'' action.
\begin{defi}
	Let $\psi$ be a colouring. We denote by $I_h(\psi)$ the two-sided ideal of $\Uha$ formed by the elements acting by zero on all the representations $\Vermah$ ($n \in \ZZ$). We denote by $\Uh$ the quotient of $\Uha$ by $I_h(\psi)$.
\end{defi}
The algebra $\Uh$ has a natural structure of $\Uha$-algebra, given by the projection map from $\Uha$ to $\Uh$, and it follows from the definition that the action of $\Uha$ on $\Vermah$ factorises through $\Uh$. The algebra $\Uh$ is universal for this property.
\begin{prop}
\label{prop:Preliminaries/Uh/universal}
	Let $\psi$ a colouring. If $A$ is a $\Uha$-algebra such that
	\begin{enumerate}
		\item the structural homomorphism from $\Uha$ to $A$ is surjective,
		\item the action of $\Uha$ on $\Vermah$ factorises through $A$ for all $n \in \ZZ$,
	\end{enumerate}
	then there is a unique surjective $\Uha$-algebra homomorphism from $A$ to $\Uh$.
\end{prop}
\begin{proof}
Let $f$ and $g$ be the structural homomorphisms from $\Uha$ to $A$ and $\Uh$, respectively. As $f$ and $g$ are surjective, a $\Uha$-algebra homomorphism from $A$ to $\Uh$ is necessarily unique and surjective. Let $x$ be an element in $\Uha$ such that $f(x) = 0$. Since the action of $\Uha$ on $\Vermah$ factorises through $A$ for all $n \in \ZZ$, it follows that $x$ acts by zero on $\Vermah$ for all $n \in \ZZ$. It implies, by definition of the algebra $\Uh$, that the image of $x$ in $\Uh$ is zero. Put in other words, the map $g$ factorises through $f$.
\end{proof}
The algebra $\Uh$ is not in general a formal deformation of $\U$. We will give in section \ref{CKM} a sufficient and necessary condition on the colouring $\psi$ for $\Uh$ to be a formal deformation of $\U$. However, the algebra $\Uh$ always satisfies the first axiom of a formal deformation; namely, $\Uh$ is topologically free.
\begin{lem}
\label{lem:Preliminaries/Uh/topologically-free}
	For any colouring $\psi$, the $\PSR$-module $\Uh$ is topologically free.
\end{lem}
\begin{proof}
The $\PSR$-module $\Uha$ is by definition topologically free. It is in particular complete (for the $h$-adic topology). As the algebra $\Uh$ is a quotient of $\Uha$, it is also complete. For $n \in \ZZ$ we denote by $E(n)$ the $\PSR$-algebra $\End_{\PSR} (\Vermah)$ and we denote by $f_n$ the $\PSR$-algebra homomorphism from $\Uha$ to $E(n)$ given by the representation $\Vermah$. We denote by $f$ the product of the $f_n$'s ($n \in \ZZ$). The ideal $I_h(\psi)$ is equal to $\ker(f)$. The $\PSR$-modules $\Vermah$ are by definition topologically free. They are in particular Hausdorff (for the $h$-adic topology) and torsion-free. Hence so is $E = \prod_{n \in \ZZ} E(n)$. As $I_h(\psi) = \ker(f)$, the algebra $\Uh$ is isomorphic to a $\PSR$-subalgebra of $E$. Therefore, $\Uh$ is Hausdorff and torsion-free. In conclusion, $\Uh$ is Hausdorff, complete and torsion-free. It is thus topologically free.
\end{proof}
Proving that $\Uh$ is a formal deformation of the $\alie$-algebra $\U$ consists from now in proving that $\hzero \Uh$ is isomorphic as an $\alie$-algebra to $\U$. As mentioned earlier, this is not true for a general colouring $\psi$. Nevertheless, we can show that the algebra $\hzero \Uh$ is always an extension of $\U$.
\begin{lem}
\label{lem:Preliminaries/Uh/projection-map}
	For any colouring $\psi$, there is a unique surjective $\alie$-algebra homomorphism from $\hzero \Uh$ to $\U$.
\end{lem}
\begin{proof}
The structural homomorphisms from $\Ua$ to the $\alie$-algebras $\hzero \Uh$ and $\U$ are surjective. It implies  that an $\alie$-algebra homomorphism from $\hzero \Uh$ to $\U$ is necessarily unique and surjective. Let us consider the functor $\hzero{(\bullet)}$ from the category of $\PSR$-modules to the category of $\field$-vector spaces. It is is a right-exact functor. Hence, there is a natural isomorphism between $\hzero \Uh$ and the quotient of $\hzero{\Uha}$ by $I_0(\psi)$, where $I_0(\psi)$ designates the image of $\hzero{I_h(\psi)}$ in $\hzero \Uha$. Using the canonical identification between $\hzero{\Uha}$ and $\Ua$, proving the lemma then reduces to proving that every element of $I_0(\psi) \subset \Ua$ is zero in $\U$. For all $n \in \ZZ$, the representation $\Vermah$ is a formal deformation along $\Uha$ of the Verma module $\Verma$, when viewed as representation of $\Ua$ (lemma \ref{lem:Preliminaries/Formal-deformations/Vermah}). This implies that every element $x \in I_0(\psi)$ acts by zero on all the integral Verma modules of $\slt$, and in consequence that $x$ is zero in $\U$ (proposition \ref{prop:Preliminaries/Uh/perfect}).
\end{proof}

\section{The equation $\psi \ltimes \xi = \theta$}
\label{Eqh}
We have established in section \ref{Preliminaries} that for every colouring $\psi$ the algebra $\Uh$ is a formal deformation of an extension of the $\alie$-algebra $\Uh$. Namely, we proved that $\Uh$ is a topologically free $\PSR$-module, and we proved that there is a surjective $\alie$-algebra homomorphism from $\hzero \Uh$ to $\U$. It follows that $\Uh$ is a formal deformation of $\U$ if and only if the aforementioned homomorphism is also injective. It is equivalent to prove that there is a relation in $\Uh$ which deforms the relation $[X^+,X^-] = H$ of $\U$, or, that the element $X^+X^-$ can be expressed in $\Uh$ as a linear combination (more precisely, as a limit of linear combinations) of the monomials $(X^-)^a (X^+)^b H^c$ ($a,b,c \geq 0$); see section \ref{CKM}. In order to address this problem, we introduce in this section infinite-dimensional linear equations which encode the action of $\Uh$ on $\Vermah$ (proposition \ref{prop:Eqh/Interpretation}). We prove that these equations always admit regular solutions if and only if the colouring $\psi$ is regular (proposition \ref{prop:Eqh/Regular-solutions}); this is the key technical result of this paper.

\subsection{Definitions and notations}
We designate by $\Coeff$ ($d \in \NN$) the $\PSR$-module formed by sequences with values in $\field^\ZZ \hhh$, of the form $f = (f^k)_{k \geq d}$ with $f^k = \sum_{m \geq 0} f^k_m h^m$ and $f^k_m \in \field^\ZZ$. We say that $f \in \Coeff$ is \emph{summable} if $f^k$ tends to zero (for the $h$-adic topology) as $k$ goes to infinity. We say that $f$ is of \emph{Verma type} if it verifies
\begin{align}
	\label{eq:Eqh/Definition/Verma1}
	& \text{$f^k(n) = 0$ for all $k \geq d$ and all $n \geq 0$, such that $n+1 \leq k \leq n+d$,}\\
	\label{eq:Eqh/Definition/Verma2}
	& \text{$f^{n+k+1}(n) = f^k(-n-2)$ for all $k \geq d$ and all $n \geq 0$.}
\end{align}
For each $d \geq 0$ we denote by $\SCoeff$ and $\VCoeff$ the $\PSR$-submodules of $\Coeff$ formed by the summable sequences, and by the sequences of Verma type, respectively. Let us remark that colourings form a subset of $\VCoeff[1]$.\par
A sequence $f$ in $\Coeff$ is said \emph{quasi-regular} if $f^k(n) \in \field[n] \hhh$ for all $k$ (i.e.\ the value $f^k_m(n) \in \field$ depends polynomially on $n$ for all $k,m$). We say that $f$ is \emph{regular} if there exists a sequence $g$ in $\SCoeff[0]$ such that $f^k(n) = \sum_{a = 0}^\infty g^a(k) n^a$ for all $k \geq d$ and all $n$ (the series is convergent as $g$ is summable). Let us remark that $f$ is regular if and only if it is quasi-regular, and if for each $m$ the degree of the polynomial $f^k_m(n)$ is a function of $k$ bounded above.\par
For $f \in \Coeff$ we denote by $f[-1]$ the sequence in $\Coeff[d+1]$ defined by $(f[-1])^k = f^{k-1}$. We denote by $[+1]$ the inverse $\PSR$-linear map, from $\Coeff[d+1]$ to $\Coeff$. Let us remark that the maps $[-1]$, $[+1]$ preserve summability and regularity. Let us furthermore remark that $[+1]$ sends sequences of Verma type to sequences of Verma type (this is not true for $[-1]$ in general). \par
Let $\psi \in \VCoeff[1]$ and let $\xi \in \Coeff$. We denote by $\psi \ltimes \xi$ the sequence in $\VCoeff$ defined for $k \geq d$ and for $n \in \ZZ$ by
\begin{equation*}
	(\psi \ltimes \xi)^k(n) \, = \, \sum_{a = d}^k \left( \prod_{b = k-a+1}^k \psi^b(n) \right) \xi^a(n-2k).
\end{equation*}
Let us note that the product $\prod_{b = k-a+1}^k \psi^b(n)$ is empty when $a$ is zero, and thus equal to one by convention.
\begin{rem}
\label{rem:Eqh/Definition}
	A sequence $f$ in $\VCoeff$ ($d \in \NN$) is regular if and only if $f^k(n) \in \field[k,n] \hhh$ (i.e.\ the value $f^k_m(n) \in \field$ depends polynomially on $k$ and $n$, for each $m$). Whereas we do not really use this fact in the present paper, the author believes that it is interesting in its own.
\end{rem}
\begin{proof}[of remark \ref{rem:Eqh/Definition}]
Using the maps $[+1]$ and $[-1]$, we can assume without loss of generality that $d=0$. Let us fix $m \geq 0$. On the one hand, it follows from the definition of regularity that there are functions $c^0, c^1, \dotsc, c^p$ ($p \geq 0$) from $\NN$ to $\field$ such that $f^k_m(n) = \sum_{a = 0}^p c^a(k) n^a$ for all $k,n$. On the other hand, $f$ being of Verma type, it follows from condition \eqref{eq:Eqh/Definition/Verma2} that $f^{n+k+1}_m(n) \in \field[n]$ for all $k \geq 0$. The sequence $(f^k_m)_{k \geq 0}$ therefore verifies the assumptions of lemma \ref{lem:Eqh/Definition} below. This proves that if $f$ is regular, then $f^k(n) \in \field[k,n] \hhh$. The converse implication is immediate.
\end{proof}
\begin{lem}
\label{lem:Eqh/Definition}
	Let $(f^k)_{k \geq 0}$ be a sequence with values in $\field^\ZZ$. We suppose that there are functions $c^0, c^1 \dotsc, c^p$ ($p \geq 0$) from $\NN$ to $\field$ such that $f^k(n) = \sum_{a=0}^p c^a(k) n^a$, and we suppose that $f^{n+k+1}(n) \in \field[n]$ for all $k \geq 0$. Then $f^k(n) \in \field[k,n]$.
\end{lem}
\begin{proof}
Let us assume by induction that the claim holds for all $p' < p$. Induction starts at $p=0$. We denote by $\delta$ the $\field$-linear endomorphism of $\field^\ZZ$ defined by $(\delta g)(n) = g(n) - g(n-1)$ for $g \in \field^\ZZ$ and $n \in \ZZ$. We denote by $\delta^p$ the $p$-th power of $\delta$. By hypothesis, $f^{n+k+1}(n) \in \field[n]$ for all $k \geq 0$. Shifting $n$ by $k$, we obtain that $f^{n+1}(n-k) \in \field[n]$ for all $k \geq 0$. We obtain in particular that the value
\begin{equation*}
	\sum_{k = 0}^p \frac{(-1)^k}{k! (p-k)!} f^{n+1}(n-k) = \frac{(\delta^p f^{n+1})(n)}{p!} = c^p(n+1)
\end{equation*}
depends polynomially on $n$. It follows that $f_\ast^{n+k+1}(n) \in \field[n]$ for all $k \geq 0$, where $f_\ast \in \field^\ZZ$ is defined for $n \in \ZZ$ by $f_\ast^k(n) = \sum_{a=0}^{p-1} c^a(k) n^k$. One concludes by using the induction hypothesis.
\end{proof}

\subsection{Interpretation}
Let $\psi$ be a colouring, and let $d \in \ZZ$. We designate by $d^+$ the non-negative integer $\max(d,0)$. We say that an element $x$ in $\Uh$ is of degree $d$ if for every $k \geq d^+$ and for every $n \in \ZZ$, $x. b_k$ is equal to $\psi(x)^k(n) b_{k-d}$ in $\Vermah$ for some $\psi(x)^k(n) \in \PSR$, and if $x.b_k$ is zero for all $0 \leq k < d^+$. Let us remark that the degree $d$ of $x$ is unique (by definition of $\Uh$ the element $x$ is zero if and only if it acts by zero on $\Vermah$ for all $n \in \ZZ$). Let us also remark that $H$ and $X^\pm$ are of degrees $0$ and $\pm 1$ in $\Uh$, respectively. As $\Uh$ is topologically generated by $H, X^-, X^+$, it then follows from colouring axioms (C2), (C3), and from the definition of $\Vermah$, that the values $\psi(x)^k(n)$ ($k \geq d^+$, $n \in \ZZ$) define a sequence $\psi(x)$ in $\VCoeff[d^+]$.
\begin{prop}
\label{prop:Eqh/Interpretation}
	Let $\psi$ be a colouring, let $d \in \ZZ$, and let $\xi \in \SCoeff[d^+]$ be a regular sequence. The series $\sum_{a \geq d^+} (X^-)^{a-d} (X^+)^a \xi^a (H)$ converges to a unique element $x$ in $\Uh$ -- for the $h$-adic topology, and the sequence $\xi$ is a solution of the equation $\psi \ltimes \xi = \psi(x)$.
\end{prop}
\begin{proof}
It follows from regularity and summability that $\xi^a(n)$ ($a \geq d^+$) define a sequence in $\field[n] \hhh$, converging to zero (for the $h$-adic topology). Therefore, $\xi^a(H)$ defines an element in $\Uh$ which tends to zero (for the $h$-adic topology), as $a$ goes to infinity. As $\Uh$ is topologically free (lemma \ref{lem:Preliminaries/Uh/topologically-free}), the series $\sum_{a \geq d^+} (X^-)^{a-d} (X^+)^a \xi^a(H)$ then converges to a unique element $x$ in $\Uh$, of degree $d$. It follows from the definition of $\psi(x)$ that
\begin{equation*}
	\sum_{a = d^+}^\infty (X^-)^{a-d} (X^+)^a \xi^a(H) . b_k \, = \, \psi(x)^k(n) b_{k-d}
\end{equation*}
holds in the representation $\Vermah$ for all $k \geq d^+$ and all $n \in \ZZ$. In other words,
\begin{equation*}
	\sum_{a = d^+}^k \left( \prod_{b = k-a+1}^k \psi^b(n) \right) \xi^a (n-2k) \, = \, \psi(x)^k(n)
\end{equation*}
holds for all $k \geq d^+$ and all $n \in \ZZ$.
\end{proof}

\subsection{Regular solutions}
Let $\psi$ be a colouring. As already mentioned, proving that $\Uh$ is a formal deformation of $\U$ amounts to proving that there is a relation in $\Uh$ which deforms the relation $[X^+,X^-] = H$ of $\U$. In particular (equivalently, in fact), we want to be able to express the element $X^+X^-$ in $\Uh$ as a linear combination (more precisely, as a limit of linear combinations) of the monomials $(X^-)^a (X^+)^b H^c$ ($a,b,c \geq 0$). In view of proposition \ref{prop:Eqh/Interpretation}, and since $\psi(X^+X^-) = \psi[+1]$, this leads us to look for regular solutions $\xi$ of $\psi \ltimes \xi = \psi[+1]$.\par
We prove here that the equation $\psi \ltimes \xi = \psi[+1]$ admits a regular summable solution $\xi$ if and only if the colouring $\psi$ is regular. We moreover prove that the equation $\psi \ltimes \xi = \theta$ has a unique regular summable solution $\xi$, for every regular sequence $\theta$ of Verma type, provided that $\psi$ is regular. This is the key technical result of this paper.
\begin{prop}
\label{prop:Eqh/Regular-solutions}
	Let $\psi$ be colouring.
	\begin{enumerate}[1.]
		\item[\emph{1.}] The equation $\psi \ltimes \xi = \psi[+1]$ admits a regular solution $\xi$ in $\SCoeff[0]$ if and only if $\psi$ is regular.
		\item[\emph{2.}] Let us suppose that $\psi$ is regular. For each $d \geq 0$ the map $\xi \mapsto \psi \ltimes \xi$ induces a $\PSR$-linear isomorphism from regular sequences in $\SCoeff$ to regular sequences in $\VCoeff$.
	\end{enumerate}
\end{prop}
\begin{proof}[of proposition \ref{prop:Eqh/Regular-solutions}] The proof of the proposition is in six steps.

\subsubsection*{\sc Step 1.} Let $\xi \in \SCoeff[0]$. If $\psi$ and $\xi$ are regular, then $\psi \ltimes \xi$ is.
\begin{proof}
Let $\theta$ designate the sequence $\psi \ltimes \xi$:
\begin{equation*}
	\theta^k(n) \, = \, \sum_{a=0}^k \left( \prod_{b = k-a+1}^k \psi^b(n) \right) \xi^a(n-2k), \ \text{for $k \geq 0$ and $n \in \ZZ$.}
\end{equation*}
We suppose that $\psi$ and $\xi$ are regular. They are in particular quasi-regular, and so then is the sequence $\theta$. Let $m \in \NN$. The sequence $\xi$ being by hypothesis summable, there is $a(m) \geq 0$ such that $\xi^a(n) \in h^{m+1} \field[n] \hhh$ for all $a > a(m)$. Therefore, the equality
\begin{equation*}
	\theta^k(n) \, = \, \sum_{a=0}^{a(m)} \left( \prod_{b = k-a+1}^k \psi^b(n) \right) \xi^a(n-2k)
\end{equation*}
holds in $\field[n] \hhh / h^{m+1} \field[n] \hhh$ for all $k \geq a(m)$. The sequence $\psi$ being by assumption regular, it follows that the degree of the polynomial $\theta^k_m(n)$ is a function of $k$ bounded above.
\end{proof}

\subsubsection*{\sc Step 2.} Let $\theta \in \VCoeff[0]$. If $\psi$ and $\theta$ are quasi-regular, then the equation $\psi \ltimes \xi = \theta$ admits a quasi-regular solution $\xi$ in $\Coeff[0]$.
\begin{proof}
We suppose that $\psi$ and $\theta$ are quasi-regular. Let $k \geq 0$ and let us assume by induction that there exist $\xi^0(n), \xi^1(n), \dotsc, \xi^{k-1}(n) \in \field[n] \hhh$ verifying
\begin{equation}
\label{eq:Eqh/Regular-solutions/Step2/1}
	\sum_{a=0}^l \left( \prod_{b = l-a+1}^l \psi^b(n) \right) \xi^a(n-2l) \, = \, \theta^l(n)
\end{equation}
for all $0 \leq l \leq k-1$. It follows from colouring axioms (C1) and (C2) that $\psi^b(n)$ is equal to $(n-b+1) f^b(n)$ for some invertible element $f^b(n)$ in $\field[n] \hhh$. Therefore, there exists $\xi^k(n)$ in $\field[n] \hhh$ such that
\begin{equation*}
	\sum_{a=0}^{k-1} \left( \prod_{b = k-a+1}^k \psi^b(n) \right) \xi^a(n-2k) \, + \, \left( \prod_{b=1}^k \psi^b(n) \right) \xi^k(n-2k) \, = \, \theta^k(n)
\end{equation*}
if only if the equality
\begin{equation}
\label{eq:Eqh/Regular-solutions/Step2/2}
	\sum_{a=0}^{k-1} \left( \prod_{b = k-a+1}^k \psi^b(n') \right) \xi^a(n'-2k) \, = \, \theta^k(n')
\end{equation}
holds in $\PSR$ for all $n' \in \{0,1,\dotsc,k-1\}$. For such $n'$, the left-hand side of \eqref{eq:Eqh/Regular-solutions/Step2/2} is equal to
\begin{align*}
	& \sum_{a=0}^{k-n'-1} \left( \prod_{b = k-a+1}^k \psi^b(n') \right) \xi^a(n'-2k) & \text{by (C2)} \\
	= \,& \sum_{a=0}^{k-n'-1} \left( \prod_{b = k-a+1}^k \psi^{b-n'-1}(-n'-2) \right) \xi^a(n'-2k) & \text{by (C3)} \\
	= \,& \sum_{a=0}^{k-n'-1} \left( \prod_{b=k-n'-a}^{k-n'-1} \psi^b(-n'-2) \right) \xi^a((-n'-2)-2(k-n'-1)).&
\end{align*}
The sequence $\theta$ being of Verma type, it follows from \eqref{eq:Eqh/Definition/Verma2} that for all $n' \in \{0,1,\dotsc,k-1\}$, equality \eqref{eq:Eqh/Regular-solutions/Step2/2} holds if and only if the following one does:
\begin{equation*}
	\sum_{a=0}^{k-n'-1} \left( \prod_{b=k-n'-a}^{k-n'-1} \psi^b(-n'-2) \right) \xi^a((-n'-2)-2(k-n'-1)) \, = \, \theta^{k-n'-1}(-n'-2).
\end{equation*}
The latter is equality \eqref{eq:Eqh/Regular-solutions/Step2/1} for $l = k-n'-1$ and $n = -n'-2$. One then concludes using the induction hypothesis.
\end{proof}

\subsubsection*{\sc Step 3.} Let $\theta \in \VCoeff[0]$. If $\psi$ and $\theta$ are regular, then the equation $\psi \ltimes \xi = \theta$ admits a regular solution $\xi$ in $\SCoeff[0]$.
\begin{proof}
We suppose that $\psi$ and $\theta$ are regular. It follows from step 2 that there exists a quasi-regular sequence $\xi$ in $\Coeff[0]$ such that $\psi \ltimes \xi = \theta$. It suffices to prove that $\xi$ is summable (a quasi-regular summable sequence is in particular regular). Let $m \geq 0$ and let us assume by induction that there is $a(m) \geq 0$ such that the polynomial $\xi_{m'}^a(n)$ is zero for all $a \geq a(m)$ and for all $m' < m$. We designate by $\tilde \theta$ the quasi-regular sequence in $\Coeff[a(m)]$ defined for $k \geq a(m)$ by
\begin{equation}
\label{eq:Eqh/Regular-solutions/Step3/1}
	\tilde \theta^k(n) \, = \, \theta^k(n) \, - \, \sum_{a=0}^{a(m)-1} \left( \prod_{b = k-a+1}^k \psi^b(n) \right) \xi^a(n-2k)
\end{equation}
Using the induction hypothesis, it then follows from $\psi \ltimes \xi = \theta$ that
\begin{equation}
\label{eq:Eqh/Regular-solutions/Step3/2}
	\sum_{a=a(m)}^k \left( \prod_{b = k-a+1}^k \psi_0^b(n) \right) \xi_m^a(n-2k) \, = \, \tilde \theta_m^k(n)
\end{equation}
for all $k \geq a(m)$. As $\psi$ and $\theta$ are by assumption both regular, it besides follows from \eqref{eq:Eqh/Regular-solutions/Step3/1} that there is $p \geq 0$ such that the degree of the polynomial $\tilde \theta^k_m(n)$ is at most $p$ for all $k$. On the other hand, according to the first colouring axiom, the polynomial $\psi_0^b(n)$ is of degree $1$ for all $b \geq 1$. Therefore, equality \eqref{eq:Eqh/Regular-solutions/Step3/2} implies, by induction on $k$, that $p(k) + k \leq p$ for all $k \geq a(m)$, where $p(k)$ designates the degree of the polynomial $\xi^k_m(n)$. It follows that $\xi^k_m(n)$ is zero for sufficiently large $k$.
\end{proof}

\subsubsection*{\sc Step 4.} Let $\theta \in \VCoeff[0]$. The equation $\psi \ltimes \xi = \theta$ admits at most one quasi-regular solution $\xi$.
\begin{proof}
Let $\xi$ be a solution in $\Coeff[0]$ of the equation $\psi \ltimes \xi = 0$:
\begin{equation}
\label{eq:Eqh/Regular-solutions/unicity}
	\sum_{a=0}^k \left( \prod_{b = k-a+1}^k \psi^b(n) \right) \xi^a(n-2k) \, = \, 0, \ \text{for all $k \geq 0$ and all $n \in \ZZ$.}
\end{equation}
We suppose that $\xi$ is quasi-regular. Let $k \geq 0$ and let us assume by induction that $\xi^{k'}$ is zero for all $k' < k$. It then follows from \eqref{eq:Eqh/Regular-solutions/unicity} that $(\prod_{b=1}^k \psi^b(n)) \xi^k(n-2k)$ is zero for all $n \in \ZZ$. On the other hand, according to the first colouring axiom, $\psi^b(n)$ is zero only if $n=b-1$. Therefore, $\xi^k(n-2k)$ is zero for infinitely many values of $n$. This implies that $\xi^k$ is zero, as $\xi^k(n) \in \field[n] \hhh$ according to the quasi-regularity assumption.
\end{proof}

\subsubsection*{\sc Step 5.} If the equation $\psi \ltimes \xi = \psi[+1]$ admits a regular solution $\xi$ in $\SCoeff[0]$, then $\psi$ is regular.
\begin{proof}
Let $\xi$ be a solution in $\Coeff[0]$ of the equation $\psi \times \xi = \psi[+1]$:
\begin{equation}
\label{eq:Eqh/Regular-solutions/Step5}
	\sum_{a=0}^k \left( \prod_{b = k-a+1}^k \psi^b(n) \right) \xi^a(n-2k) \, = \, \psi^{k+1}(n), \ \text{for all $k \geq 0$ and all $n \in \ZZ$.}
\end{equation}
We suppose that $\xi$ is regular and summable. It is in particular quasi-regular, and it follows from \eqref{eq:Eqh/Regular-solutions/Step5}, by induction on $k$, that $\psi$ is quasi-regular.\par
Let us recall that $\Ncol = (\Ncol^k)_{k \geq 1}$ designates the colouring with values in $\field^\ZZ$, defined by $\Ncol^k(n) = k(n-k+1)$ for $n \in \ZZ$. There is an evident quasi-regular solution $f$ in $\Coeff[0]$ of the equation $\Ncol \ltimes f = \Ncol[+1]$, which is given by $f^0(n) = n$, $f^1(n) = 1$ and $f^k(n) = 0$ for $k \geq 2$.  This is the unique quasi-regular solution (step 4). On the other hand, it follows from the first colouring axiom that $(\xi^k_0)_{k \geq 0}$ is another solution, also quasi-regular, since $\xi$ is by assumption. This proves that $\xi_0^0(n) = n$, $\xi_0^1(n) = 1$ and $\xi^k_0(n) = 0$ for all $k \geq 2$.\par
Let $m \geq 0$ and let us assume by induction that there is $p \geq 0$ such that the degree of the polynomial $\psi^k_{m'}(n)$ is at most $p$ for all $k \geq 1$ and for all $m' < m$. As $\xi$ is by assumption summable, there is $a(m) \geq 0$ such that $\xi^a(n) \in h^{m+1} \field[n] \hhh$ for all $a > a(m)$. Let $p' \geq 0$ such that the degree of the polynomial $\xi^a_{m'}(n)$ is at most $p'$ for all $a \leq a(m)$ and all $m' \leq m$. The following equation holds in $\PSR / h^{m+1} \PSR$ for all $k \geq a(m)$:
\begin{equation*}
	\sum_{a=0}^{a(m)} \left( \prod_{b = k-a+1}^k \psi^b(n) \right) \xi^a(n-2k) \, = \, \psi^{k+1}(n).
\end{equation*}
Using that $\xi^1_0(n) = 1$, that $\xi^k_0(n) = 0$ for $k \geq 2$, and using the induction hypothesis, it follows that for all $k \geq a(m)$, $\psi_m^{k+1}(n) = \xi^0_m(n-2k) + \psi_m^k(n) + g^k(n)$ for some $g^k(n) \in \field[n]$ of degree at most $a(m) p + p'$. This proves, by induction on $k$, that the degree of the polynomial $\psi^k_m(n)$ is a function of $k$ bounded above
\end{proof}

\subsubsection*{\sc Step 6.} For $f \in \Coeff$ ($d \in \NN$) let $f\{-1\}$ designate the sequence in $\Coeff[d+1]$ defined by $(f\{-1\})^k(n) = \psi^k(n) f^{k-1}(n)$ for $k \geq d+1$ and $n \in \ZZ$.
\begin{enumerate}[\sc i.]
\item The equality $\psi \ltimes (\xi[-1]) = (\psi \ltimes \xi)\{-1\}$ holds for all $\xi \in \Coeff$.
\item Let us suppose that $\psi$ is regular. The map $f \mapsto f\{-1\}$ induces a $\PSR$-linear isomorphism from regular sequences in $\VCoeff$ to regular sequences in $\VCoeff[d+1]$.
\end{enumerate}
\begin{proof}
Point {\sc i} is straightforward calculations. Let us prove point {\sc ii}. We suppose that $\psi$ is regular. It follows from colouring axioms (C2) and (C3) that if $f$ is of Verma type, then $f\{-1\}$ is. Also, as the colouring $\psi$ is by assumption regular, $f$ regular implies $f\{-1\}$ regular. This proves that $f \mapsto f\{-1\}$ defines a $\PSR$-linear map, from regular sequences in $\VCoeff$ to regular sequences in $\VCoeff[d+1]$. Let $\theta$ be a regular sequence in $\VCoeff[d+1]$. Then $n-k$ divides $\theta^{k+1}(n)$ in $\field[n] \hhh$ for all $k \geq d$. On the other hand, it follows from colouring axioms (C1) and (C2), and as $\psi$ is by assumption regular, that for all $k \geq 0$, $\psi^{k+1}(n) = (k+1)(n-k) g^k(n)$ for some unique invertible element $g^k(n)$ in $\field[n] \hhh$. Therefore, for all $k \geq d$, there is a unique element $(\theta\{+1\})^k(n)$ in $\field[n] \hhh$ such that $\theta^{k+1}(n) = \psi^{k+1}(n) (\theta\{+1\})^k(n)$. In other words, there is a unique quasi-regular sequence $\theta\{+1\}$ in $\Coeff$ such that $\theta = (\theta\{+1\})\{-1\}$. Let $\tilde g^k(n)$ ($k \geq 0$) designates the inverse of $g^k(n)$ in $\field[n] \hhh$. It follows from the definition of $\theta\{+1\}$ that
\begin{equation}
\label{eq:Eqh/Regular-solutions/Step6}
	(k+1)(n-k) (\theta\{+1\})^k(n) = \tilde g^k(n) \theta^{k+1}(n), \ \text{for all $k \geq d$.}
\end{equation}
As the colouring $\psi$ is regular, so then is the sequence $\tilde g = (\tilde g^k)_{k \geq 0}$. As $\theta$ is also regular, it follows from \eqref{eq:Eqh/Regular-solutions/Step6} that $\theta\{+1\}$ is. Let us fix $k \geq d$ and $n \geq 0$. As $\theta$ is of Verma type, it follows from \eqref{eq:Eqh/Regular-solutions/Step6} that $(k+1)(n-k) (\theta\{+1\})^k(n)$ is zero if $n+1 \leq k+1 \leq n+d+1$. In particular, $(\theta\{+1\})^k(n)$ is zero if $n+1 \leq k \leq n+d$. Colouring axiom (C3) implies that $\tilde g^{n+k+1}(n) = \tilde g^k(-n-2)$. Using that $\theta$ is of Verma type, it then follows from \eqref{eq:Eqh/Regular-solutions/Step6} that $(\theta\{+1\})^{n+k+1}(n) = (\theta\{+1\})^k(-n-2)$. We thus have proved that $\theta \mapsto \theta\{+1\}$ defines a map from regular sequences in $\VCoeff[d+1]$ to regular sequences in $\VCoeff$, and that $\{+1\}$ is a right inverse of the map $\{-1\}$. It follows by definition that $\{+1\}$ is also a left inverse.
\end{proof}

\subsubsection*{Conclusion.} Step 3 for $\theta = \psi[+1]$, together with step 5, prove point 1 of the proposition. Let us suppose that $\psi$ is regular. Step 1 then implies that $\xi \mapsto \psi \ltimes \xi$ defines a $\PSR$-linear map from regular sequences in $\SCoeff[0]$ to regular sequences in $\VCoeff[0]$. Steps 3 and 4 prove that the map is surjective and injective, respectively. This establishes point 2 of the proposition for $d = 0$. The general case follows, by induction on $d$, from step 6.
\end{proof}
\begin{rem}
	Let $\psi$ a colouring and let $\theta \in \VCoeff$ ($d \in \NN$). We do not assume here that $\psi$ and $\theta$ are regular. We see from the proof of proposition \ref{prop:Eqh/Regular-solutions} (steps 2 and 6) that the equation $\psi \ltimes \xi = \theta$ admits as many solutions $\xi$ in $\Coeff$ as they are choices for the values $\xi^{k+d}(-2k), \dotsc, \xi^{k+d}(-k-1)$ ($k \geq 0$). We see in particular that if not regular, a solution is not unique. It may be interesting to find an explicit condition on those values of $\xi$ characterising regularity for the solution $\xi$, when $\psi$ and $\theta$ are regular. 
\end{rem}

\section{Coloured Kac-Moody algebras of rank one}
\label{CKM}
We present here the main results of this paper. We prove that $\Uh$ is a formal deformation of $\U$ if and only if the colouring $\psi$ is regular (theorem \ref{thm:CKM/Formal-deformations}). We give a Chevalley-Serre presentation of $\Uh$ for $\psi$ regular (theorem \ref{thm:CKM/Generators-and-relations}). We show that the constant formal deformation $\U \hhh$ and the quantum algebra $\Uq$ can both be realized as coloured Kac-Moody algebras (theorem \ref{thm:CKM/Realizations}). We prove that coloured Kac-Moody algebras are $\blie$-trivial deformations of $\U$, and admit unique $\blie$-trivializations (theorem \ref{thm:CKM/b-triviality}). We prove that regular colourings classify $\hlie$-trivial formal deformations of the $\alie$-algebra $\U$ (theorem \ref{thm:CKM/Classification}). As a corollary, we obtain a rigidity result for $\U$; namely, every $\hlie$-trivial formal deformation of the $\alie$-algebra $\U$ is also $\blie$-trivial, and admits a unique $\blie$-trivialization (corollary \ref{cor:CKM/b-triviality}).

\subsection{Formal deformations of $\U$}
We recall that a colouring $\psi$ is said regular if $\psi^k(n) \in \field[k,n] \hhh$ (see remark \ref{rem:Eqh/Definition}). For $\psi$ regular, we call the algebra $\Uh$ a \emph{coloured Kac-Moody algebra}.
\begin{thm}
\label{thm:CKM/Formal-deformations}
	Let $\psi$ be a colouring. The following three assertions are equivalent.
	\begin{enumerate}
		\item The colouring $\psi$ is regular.
		\item The algebra $\Uh$ is a formal deformation of the $\field$-algebra $\U$.
		\item The algebra $\Uh$ is a formal deformation of the $\alie$-algebra $\U$.
	\end{enumerate}
\end{thm}
\begin{proof*}[of theorem \ref{thm:CKM/Formal-deformations}]\par
Let $f$ be the surjective $\alie$-algebra homomorphism from $\hzero \Uh$ to $\U$ (lemma \ref{lem:Preliminaries/Uh/projection-map}). We already proved that $\Uh$ is topologically free (lemma \ref{lem:Preliminaries/Uh/topologically-free}), it is therefore sufficient to prove that $f$ is injective in order to prove that $\Uh$ is a formal deformation of the $\alie$-algebra $\U$. We propose to prove the following implications: (i) $\Rightarrow$ (iii), (ii) $\Rightarrow$ (iii) and (iii) $\Rightarrow$ (i). The implication (iii) $\Rightarrow$ (ii) is immediate.

\subsubsection*{1.} Assertion (i) implies assertion (iii).
\begin{proof}
Let us suppose that $\psi$ is regular. Then, the equation $\psi \ltimes \xi = \psi[+1]$ admits a regular solution $\xi$ in $\SCoeff[0]$ (proposition \ref{prop:Eqh/Regular-solutions}). The series $\sum_{a \geq 0} (X^-)^a (X^+)^a \xi^a(H)$ converges to a unique element $x$ in $\Uh$ (for the $h$-adic topology), such that $\psi(x) = \psi[+1]$ (proposition \ref{prop:Eqh/Interpretation}). As $\psi[+1] = \psi(X^+X^-)$, this means that for every $n \in \ZZ$, the elements $X^+X^-$ and $x$ act identically on $\Vermah$, when viewed as a representation of $\Uh$. Hence, by definition of $\Uh$, the relation $X^+X^- = x$ holds in $\Uh$. It follows that $\hzero \Uh$ is spanned by the monomials $(X^-)^a (X^+)^b H^c$ ($a,b,c \geq 0$). On the other hand, these monomials form the PBW basis of $\U$. In other words, $f$ sends a spanning subset to a basis. This implies that $f$ is injective.
\end{proof}

\subsubsection*{2.} Assertion (ii) implies assertion (iii).
\begin{proof}
We suppose that $\Uh$ is a formal deformation of the $\field$-algebra $\U$. Then there is a $\field$-algebra isomorphism $f'$ from $\U$ to $\hzero \Uh$. Let $g$ designate the map $f \circ f'$, it a surjective $\field$-algebra endomorphism of $\U$. We denote by $\Simple^g$ the pullback by $g$ of the $(n+1)$-dimensional irreducible representation $\Simple$ of $\slt$ ($n \geq 0$). The pullback $\Simple^g$ is a representation of $\U$ of dimension $n+1$, irreducible again, as $g$ is surjective. The representation $\Simple^g$ is thus isomorphic to $\Simple$. Let $x$ in $\U$ such that $g(x) = 0$. The element $x$ acts by zero on the pullback $\Simple^g$ for every $n \geq 0$. As a consequence, $x$ acts by zero on $\Simple$ for every $n$, and is thus equal to zero (proposition \ref{prop:Preliminaries/Uh/perfect}). In other words $g$ is injective. It implies that $f$ is injective.
\end{proof}

\subsubsection*{3.} Assertion (iii) implies assertion (i).
\begin{proof}
Let $\Uh[\psi,0]$ designate the $\PSR$-submodule of $\Uh$ formed by the elements $x$ such that $[H,x] = 0$. We denote by $P$ the subset of $\Uh[\psi,0]$ formed by the monomials $(X^-)^a (X^+)^a H^b$ ($a,b \geq 0$). The $\PSR$-submodule $\Uh[\psi,0]$ is closed in $\Uh$ (for the $h$-adic topology). Hence, as $\Uh$ is topologically free (lemma \ref{lem:Preliminaries/Uh/topologically-free}), so then is $\Uh[\psi,0]$, and the inclusion map from $P$ to $\Uh[\psi,0]$ induces a $\PSR$-linear map $j$ from $(\field P) \hhh$ to $\Uh[\psi,0]$. We denote by $f$ the inclusion map from $\Uh[\psi,0]$ to $\Uh$. The map $\hzero f$ is injective: $[H,hx] = 0$ implies $[H,x] = 0$ for any $x \in \Uh$. Let us suppose that $\Uh$ is a formal deformation of the $\alie$-algebra $\U$. Then, there exists an $\alie$-algebra isomorphism $g_0$ from $\hzero \Uh$ to $\U$. The map $g_0 \circ \hzero f$ induces an injective map $f_0$ from $\hzero{\Uh[\psi,0]}$ to $U(0)$, where $U(0)$ designates the subspace of $\U$ formed by the elements $x$ such that $[H,x] = 0$. In view of the PBW basis of $\U$, the monomials $(X^-)^a (X^+)^a H^b$ ($a,b \geq 0$) form a basis of $U(0)$. This implies that $f_0 \circ \hzero j$ is surjective. As $f_0$ is injective, it follows that $\hzero j$ is surjective. Since $(\field P) \hhh$ and $\Uh$ are both topologically free, $j$ is surjective as well. The element $X^+ X^-$ in $\Uh[\psi,0]$ therefore belongs to the image of $j$. In other words, there exists a regular sequence $\xi = (\xi^a)_{a \geq 0}$ in $\SCoeff[0]$ such that the element $X^+ X^-$ is equal to $\sum_{a=0}^\infty (X^-)^a (X^+)^a \xi^a(H)$ in $\Uh$. This implies that $\xi$ is a solution of $\psi \ltimes \xi = \psi(X^+X^-)$ (proposition \ref{prop:Eqh/Interpretation}). As $\psi(X^+X^-) = \psi[+1]$, this proves that the equation $\psi \ltimes \xi = \psi[+1]$ admits a regular solution in $\SCoeff[0]$. Therefore, the colouring $\psi$ is regular (proposition \ref{prop:Eqh/Regular-solutions}).
\end{proof}
\end{proof*}

\subsection{Generators and relations}
We give here a Chevalley-Serre presentation for the coloured Kac-Moody algebra $\Uh$.
\begin{thm}
\label{thm:CKM/Generators-and-relations}
	Let $\psi$ be a regular colouring. The $\PSR$-algebra $\Uh$ is topologically generated by $H,X^-, X^+$ and subject to the relations
	\begin{subequations}
	\label{eq:CKM/Generators-and-relations}
		\begin{align}
			\label{eq:CKM/Generators-and-relations/non-deformed}
			[H, X^\pm] \, &= \, \pm 2 X^\pm, \\
			\label{eq:CKM/Generators-and-relations/deformed}
			X^+X^- \, &= \, \sum_{a=0}^\infty (X^-)^a (X^+)^a \xi^a(H),
		\end{align}
	\end{subequations}
	where $\xi$ designates the regular solution in $\SCoeff[0]$ of the equation $\psi \ltimes \xi = \psi[+1]$ (proposition \ref{prop:Eqh/Regular-solutions}).
\end{thm}
\begin{proof}
The element $\xi^a(H)$ tends to zero in $\field[H] \hhh$ (for the $h$-adic topology), as $a$ goes to infinity. This implies that the the right-hand side of \eqref{eq:CKM/Generators-and-relations/deformed} converges to a unique element in the $\PSR$-algebra $\field \langle H,X^-,X^+ \rangle \hhh$. Let then $U$ be the quotient of $\field \langle H,X^-,X^+ \rangle \hhh$ by the smallest closed (for the $h$-adic topology) two-sided ideal containing relations \eqref{eq:CKM/Generators-and-relations}. The sequence $\xi$ being a regular solution in $\SCoeff[0]$ of the equation $\psi \ltimes \xi = \psi[+1]$, the series $\sum_{a \geq 0} (X^-)^a (X^+)^a \xi^a(H)$ converges to a unique element $x$ in $\Uh$ (for the $h$-adic topology), such that $\psi(x) = \psi[+1]$ (proposition \ref{prop:Eqh/Interpretation}). As $\psi[+1] = \psi(X^+X^-)$, this means that for every $n \in \ZZ$, the elements $X^+X^-$ and $x$ act identically on $\Vermah$, when viewed as a representation of $\Uh$. Hence, by definition of $\Uh$, relation \eqref{eq:CKM/Generators-and-relations/deformed} holds in $\Uh$. It follows, as $\Uh$ is topologically free (lemma \ref{lem:Preliminaries/Uh/topologically-free}), that there is a canonical $\field$-algebra homomorphism $f$ from $U$ to $\Uh$. On the other hand, relations \eqref{eq:CKM/Generators-and-relations} imply that $\hzero U$ is spanned by the monomials $(X^-)^a (X^+)^b H^c$ ($a,b,c \geq 0$). These monomials form the PBW basis of $\U$, and of $\hzero \Uh$ as well, since $\U$ and $\hzero \Uh$ are isomorphic as $\alie$-algebras (theorem \ref{thm:CKM/Formal-deformations}). In other words, the map $\hzero f$ sends a spanning subset to a basis. The map $\hzero f$ is therefore bijective. Since $U$ is Hausdorff and complete (for the $h$-adic topology), and since $\Uh$ is topologically free (lemma \ref{lem:Preliminaries/Uh/topologically-free}), it follows that $f$ is bijective.
\end{proof}

\subsection{Classical and quantum realizations}
We prove here that the constant formal deformation $\U \hhh$ and the quantum algebra $\Uq$ can both be realized as coloured Kac-Moody algebras.\par
We recall that the quantum algebra $\Uq$ is the $\Uha$-algebra topologically generated by $H,X^-, X^+$ and subject to the relation
\begin{equation}
\label{eq:CKM/Realizations/Drinfel'd-Jimbo}
	[X^+, X^-] \, = \, \frac{q^H  - q^{-H}}{q - q^{-1}} \quad \text{with $q = \exp(h)$ and $q^H = \exp(hH)$,}
\end{equation}
i.e.\ $\Uq$ is the quotient of $\Uha$ by the smallest closed (for the $h$-adic topology) two-sided ideal containing relation \eqref{eq:CKM/Realizations/Drinfel'd-Jimbo}.
\begin{thm}
\label{thm:CKM/Realizations}
	As $\Uha$-algebras, $\U \hhh$ and $\Uq$ are isomorphic to the coloured Kac-Moody algebras $\Uh[\Ncol]$ and $\Uh[\qcol]$, respectively.
\end{thm}
\begin{proof}
Let us recall that $\Ncol$ and $\qcol$ are the colourings defined by $\Ncol^k(n) = k(n-k+1)$ and $\qcol^k(n) = [k]_q [n-k+1]_q$ (for $k \geq 1$ and for $n \in \ZZ$).\par
The relation $[X^+,X^-] = H$ holds in the representation $\Vermah[n,\Ncol]$ for all $n \in \ZZ$. It follows, in view of the Chevalley-Serre presentation of $\U$, and as $\Vermah[n,\Ncol]$ is Hausdorff (for the $h$-adic topology), that the action of $\Uha$ on $\Vermah[n,\Ncol]$ factorises through $\U \hhh$ for all $n \in \ZZ$. This implies, by the universal property of $\Uh[\Ncol]$ (proposition \ref{prop:Preliminaries/Uh/universal}), that there exists a surjective $\Uha$-algebra homomorphism $f$ from $\U \hhh$ to $\Uh[\Ncol]$. On the other hand, the $\alie$-algebra $\hzero{\U \hhh}$ is isomorphic to $\U$. Hence, there is an $\alie$-algebra homomorphism $g_0$ from $\hzero {\Uh[\Ncol]}$ to $\hzero{\U \hhh}$ (lemma \ref{lem:Preliminaries/Uh/projection-map}). Let us then consider the map $g_0 \circ \hzero f$. It is an $\alie$-algebra endomorphism of $\hzero{\U \hhh}$. Therefore, $g_0 \circ \hzero f$ is equal to the identity map. This implies that $\hzero f$ is injective. Since $\Uh[\Ncol]$ is torsion-free (lemma \ref{lem:Preliminaries/Uh/topologically-free}), and since $\U \hhh$ is Hausdorff (for the $h$-adic topology), it follows that $f$ is injective, and thus bijective.\par
The proof for $\Uq$ is similar. Namely, relation \eqref{eq:CKM/Realizations/Drinfel'd-Jimbo} holds in the representation $\Vermah[n,\qcol]$ for all $n \in \ZZ$. It follows, as $\Vermah[n,\qcol]$ is Hausdorff (for the $h$-adic topology), that the action of $\Uha$ on $\Vermah[n,\qcol]$ factorises through $\Uq$ for all $n \in \ZZ$. This implies, by the universal property of $\Uh[\qcol]$ (proposition \ref{prop:Preliminaries/Uh/universal}), that there exists a surjective $\Uha$-algebra homomorphism $f$ from $\Uq$ to $\Uh[\qcol]$. On the other hand, the $\alie$-algebra $\hzero \Uq$ is isomorphic to $\U$ (the functor $\hzero{(\bullet)}$ from the category of $\PSR$-modules to the category of $\field$-vector spaces is a right-exact functor, and relation \eqref{eq:CKM/Realizations/Drinfel'd-Jimbo} is $[X^+,X^-] = H$ modulo $h$). Hence, there is an $\alie$-algebra homomorphism $g_0$ from $\hzero{\Uh[\qcol]}$ to $\hzero \Uq$ (lemma \ref{lem:Preliminaries/Uh/projection-map}). Let us then consider the map $g_0 \circ \hzero f$. It is an $\alie$-algebra endomorphism of $\hzero \Uq$. Therefore, $g_0 \circ \hzero f$ is equal to the identity map. This implies that $\hzero f$ is injective. Since $\Uh[\qcol]$ is torsion-free (lemma \ref{lem:Preliminaries/Uh/topologically-free}), and since $\Uq$ is Hausdorff (for the $h$-adic topology), it follows that $f$ is injective, and thus bijective.
\end{proof}

\subsection{$\blie$-triviality}
Let $\psi$ be a regular colouring. We know that the coloured Kac-Moody algebra $\Uh$ is a formal deformation of the $\alie$-algebra $\U$ (theorem \ref{thm:CKM/Formal-deformations}). We say that the formal deformation $\Uh$ is \emph{$\blie$-trivial} if there is a $\PSR$-algebra isomorphism $g$ from $\U \hhh$ to $\Uh$, such that $g(H) = H$, $g(X^-) = X^-$, and such that $\hzero g$ is an $\alie$-algebra isomorphism. The isomorphism $g$ is called a \emph{$\blie$-trivialization} of $\Uh$.
\begin{thm}
\label{thm:CKM/b-triviality}
	Let $\psi$ be a regular colouring. The coloured Kac-Moody algebra $\Uh$ is $\blie$-trivial, and it admits a unique $\blie$-trivialization.
\end{thm}
\begin{proof}[of theorem \ref{thm:CKM/b-triviality}] The proof of the theorem is in two steps.

\subsubsection*{\sc Step 1.} There exists a surjective $\PSR$-algebra homomorphism $g$ from $\U \hhh$ to $\Uh$ such that $g(H) = H$ and $g(X^-) = X^-$.
\begin{proof}
Let us consider the natural colouring $\Ncol$, defined by $N^k(n) = k(n-k+1)$ for $k \geq 1$ and for $n \in \ZZ$. As $\Ncol$ and $\psi$ are regular, the equation $\Ncol \ltimes \xi = \psi$ admits a regular solution $\xi$ in $\SCoeff[1]$ (proposition \ref{prop:Eqh/Regular-solutions}). The series $\sum_{a \geq 1} (X^-)^{a-1} (X^+)^a \xi^a(H)$ converges to a unique element $x$ in $\Uh[\Ncol]$ (for the $h$-adic topology), such that $\Ncol(x) = \psi$ (proposition \ref{prop:Eqh/Interpretation}). We denote by $f$ be the $\PSR$-algebra homomorphism from $\Uha$ to $\Uh[\Ncol]$ defined by $f(H) = H$, $f(X^-) = X^-$, $f(X^+) = x$. In view of the first colouring axiom, $\xi_0 = (\xi^k_0)_{k \geq 1}$ is a solution of the equation $\Ncol \ltimes \xi_0 = \Ncol$. Therefore, $\xi_0^1 = 1$, and $\xi_0^k = 0$ for all $k \geq 2$, since the equation admits a unique regular solution (proposition \ref{prop:Eqh/Regular-solutions}). This proves that the images of $x$ and $X^+$ in $\hzero{\Uh[\Ncol]}$ are equal, and thus that $\hzero f$ is surjective. Since $\Uha$ is by definition topologically free, and since $\Uh[\Ncol]$ is as well (lemma \ref{lem:Preliminaries/Uh/topologically-free}), $f$ is also surjective. It follows from $\Ncol(x) = \psi$ that for all $n \in \ZZ$, the pullback by $f$ of the representation $\Vermah[n,\Ncol]$, when viewed as a representation of $\Uh[\Ncol]$, is equal to the representation $\Vermah$. In other words, the action of $\Uha$ on $\Vermah$ factorises through $f$ for all $n \in \ZZ$. It then follows from the universal property of $\Uh$ (proposition \ref{prop:Preliminaries/Uh/universal}) that there exists a surjective $\Uha$-algebra homomorphism $g$ from $\Uh[\Ncol]$ to $\Uh$, where $\Uh[\Ncol]$ is endowed with the $\Uha$-algebra structure defined by $f$. In particular, $g$ satisfies $g(H) = H$ and $g(X^-) = X^-$. As $\U \hhh$ and $\Uh[\Ncol]$ are isomorphic as $\Uha$-algebras (theorem \ref{thm:CKM/Realizations}), this concludes the proof of step 1.
\end{proof}

\subsubsection*{\sc Step 2.} The identity map is the unique $\PSR$-algebra endomorphism of $\U \hhh$ which fixes both $H$ and $X^-$.
\begin{proof}
Let $g$ be a $\PSR$-algebra endomorphism of $\U \hhh$ which fixes both $H$ and $X^-$. Let then $x^+$ be the image of $X^+$ by $g$. The relations $[H,X^+] = 2X^+$ and $[X^+,X^-] = H$ hold in $\U \hhh$, hence so do the relations $[H,x^+] = 2x^+$ and $[x^+,X^-] = H$. Let $n \in \ZZ$, and let us consider the representation $\Verma \hhh$ of $\U \hhh$. The relation $[H,x^+] = 2x^+$ implies $H.(x^+.b_0) = (n+2)(x^+.b_0)$, and thus $x^+.b_0 = 0$. It then follows, by induction $k$, and using the relation $[x^+,X^-] = H$, that $x^+.b_k = k(n-k+1)b_{k-1}$ for all $k \geq 1$. Therefore, $x^+ - X^+$ acts by zero on $\Verma \hhh$ for all $n \in \ZZ$. It follows from proposition \ref{prop:Preliminaries/Uh/perfect} that $x^+ - X^+$ is zero.
\end{proof}

\subsubsection*{Conclusion.} The unicity of a $\blie$-trivialization for $\Uh$ follows from steps 1 and 2. It remains to prove that $\Uh$ is $\blie$-trivial. Let $g$ be a $\PSR$-algebra homomorphism from $\U \hhh$ to $\Uh$, such that $g(H) = H$ and $g(X^-) = X^-$ (step 1). The map $\hzero g$ satisfies in particular $\hzero g(H) = H$ and $\hzero g(X^-) = X^-$. We denote by $\tilde g$ the $\PSR$-algebra homomorphism induced by $\hzero g$ from $(\hzero{\U \hhh}) \hhh$ to $(\hzero{\Uh}) \hhh$. The $\Uha$-algebra $(\hzero{\U \hhh}) \hhh$ is canonically isomorphic to $\U \hhh$. The $\Uha$-algebra $(\hzero{\Uh}) \hhh$ is also isomorphic to $\U \hhh$, since $\Uh$ is a formal deformation of the $\alie$-algebra $\U$ (theorem \ref{thm:CKM/Formal-deformations}). Step 2 therefore implies that $\tilde g$ is a $\Uha$-algebra isomorphism. This proves that $\hzero g$ is an $\alie$-algebra isomorphism. In particular, $\hzero g$ is bijective. Since $\Uh$ is topologically free (lemma \ref{lem:Preliminaries/Uh/topologically-free}), it follows that $g$ is bijective as well.
\end{proof}

\subsection{Classification}
Let $A$ be a formal deformation of the $\alie$-algebra $\U$. We designate again by $H,X^-,X^+$ the images of $H,X^-,X^+$, by the structural homomorphism from $\Uha$ to $A$. We say that the formal deformation $A$ is \emph{$\hlie$-trivial} if there is a $\PSR$-algebra isomorphism $g$ from $\U \hhh$ to $A$, such that $g(H) = H$, and such that $\hzero g$ is an $\alie$-algebra isomorphism. The isomorphism $g$ is called a \emph{$\hlie$-trivialization} of $A$.\par
It follows from theorem \ref{thm:CKM/b-triviality} that coloured Kac-Moody algebras are in particular $\hlie$-trivial formal deformations of $\U$. We establish in the following theorem that up to $\alie$-algebra isomorphism, there is no other $\hlie$-trivial formal deformations of $\U$, and that regular colourings classify such deformations.
\begin{thm}
\label{thm:CKM/Classification}
	For every $\hlie$-trivial formal deformation $A$ of the $\alie$-algebra $\U$, there is a unique regular colouring $\psi$ such that $A$ and $\Uh$ are isomorphic as $\Uha$-algebras.
\end{thm}
We thus obtain, in view of theorem \ref{thm:CKM/b-triviality}, the following rigidity result for $\U$.
\begin{cor}
\label{cor:CKM/b-triviality}
	A $\hlie$-trivial formal deformation of the $\alie$-algebra $\U$ is also $\blie$-trivial, and admits a unique $\blie$-trivialization.
\end{cor}
\begin{proof}[of theorem \ref{thm:CKM/Classification}]
Let us adapt the definition of the representation $\Vermah$. Namely, for $n \in \ZZ$ and for $\varphi = (\varphi^k)_{k \geq 1}$ any sequence with values in $\PSR$, we denote by $\Vermah[n,\varphi]$ the representation of $\Uha$, whose underlying $\PSR$-module is $(\bigoplus_{k  \geq 0} \field b_k) \hhh$, and where the action of $\Uha$ is given by
\begin{equation*}
	\begin{aligned}
		H . b_k \, &= \, (n-2k) b_k, \\
		X^- . b_k \, &= \, b_{k+1}, \\
		X^+ . b_k \, &= \,
		\begin{cases}
			0& \text{if $k = 0$,} \\
			\varphi^k b_{k-1} & \text{if $k \geq 1$.}
		\end{cases}
	\end{aligned}
\end{equation*}
The proof of the theorem is in three steps.

\subsubsection*{\sc Step 1.} Let $A$ be a formal deformation of the $\alie$-algebra $\U$, and let $n \in \ZZ$. There is at most one sequence $\varphi$ with values in $\PSR$ such that the action of $\Uha$ on $\Vermah[n,\varphi]$ factorises through $A$.
\begin{proof}
Let $V(n)$ be the representation of $A$ topologically generated by $v$ and subject to the relations $H.v = nv$, $X^+.v = 0$, i.e.\ the representation $V(n)$ is the quotient of the left regular representation $A$ by the smallest closed (for the $h$-adic topology) subrepresentation containing $H - n$ and $X^+$. Let $\varphi$ be a sequence with values in $\PSR$ such that the action of $\Uha$ on $\Vermah[n,\varphi]$ factorises through $A$. We regard from now $\Vermah[n,\varphi]$ as a representation of $A$. By definition of $V(n)$, and since $\Vermah[n,\varphi]$ is Hausdorff (for the $h$-adic topology), there exists an $A$-morphism $f$ from $V(n)$ to $\Vermah[n,\varphi]$ such that $f(v) = b_0$. Let $v_0$ be the image of $v$ in $\hzero{V(n)}$. The representation $\hzero{V(n)}$ of $\hzero A$ is generated by $v_0$, and verifies $H.v_0 = nv_0$, $X^+.v_0 = 0$. On the other hand, the algebra $A$ being by assumption a formal deformation of the $\alie$-algebra $\U$, the monomials $(X^-)^a (X^+)^b H^c$ ($a,b,c \geq 0$) span $\hzero A$. It then follows that the vectors $(X^-)^a.v_0$ ($a \geq 0$) span $\hzero V$. As a consequence, $\hzero f$ sends a spanning subset to a basis. The map $\hzero f$ is therefore bijective. As $V(n)$ is Hausdorff and complete (for the $h$-adic topology), and since $\Vermah[n,\varphi]$ is topologically free, it follows that $f$ is bijective. Let us now suppose that there is another sequence $\varphi'$ with values in $\PSR$ such that the action of $\Uha$ on $\Vermah[n,\varphi']$ factorises through $A$. Then, as for $\varphi$, there is an $A$-isomorphism $f'$ from $V(n)$ to $\Vermah[n,\varphi']$ such that $f'(v) = b_0$. It follows that there is an $A$-isomorphism $g$ from $\Vermah[n,\varphi]$ to $\Vermah[n,\varphi']$ such that $g(b_0) = b_0$. As $g$ commutes with the action of $X^-$, $g(b_k) = b_k$ for all $k \geq 0$. Since $g$ also commutes with the action of $X^+$, it follows that $\varphi$ are $\varphi'$ are equal.
\end{proof}

\subsubsection*{\sc Step 2.} Let $A$ be a formal deformation of the $\alie$-algebra $\U$. If $\psi$ is a colouring such that the action of $\Uha$ on $\Vermah$ factorises through $A$ for all $n \in \ZZ$, then $A$ and $\Uh$ are isomorphic as $\Uha$-algebras.
\begin{proof}
Let $\psi$ be a colouring such that the action of $\Uha$ on $\Vermah$ factorises through $A$ for all $n \in \ZZ$. Then, by the universal property of $\Uh$ (proposition \ref{prop:Preliminaries/Uh/universal}), there exists a surjective $\Uha$-algebra homomorphism $f$ from $A$ to $\Uh$ (we recall that as $A$ is by assumption a formal deformation of the $\alie$-algebra $\U$, the structural homomorphism from $\Uha$ to $A$ is surjective). On the other hand, $\hzero A$ and $\U$ being isomorphic as $\alie$-algebras, there is an $\alie$-algebra homomorphism $g_0$ from $\hzero \Uh$ to $\hzero A$ (lemma \ref{lem:Preliminaries/Uh/projection-map}). Let us then consider the map $g_0 \circ \hzero{f}$. It is an $\alie$-algebra endomorphism of $\hzero A$. Therefore, $g_0 \circ \hzero{f}$ is equal to the identity map. This implies that $\hzero{f}$ is injective. Since $A$ is by assumption topologically free, and since $\Uh$ is as well (lemma \ref{lem:Preliminaries/Uh/topologically-free}), it follows that $f$ is injective, and thus bijective.
\end{proof}

\subsubsection*{\sc Step 3.} For every $\hlie$-trivial formal deformation $A$ of the $\alie$-algebra $\U$, there exists a regular colouring $\psi$ such that $A$ and $\Uh$ are isomorphic as $\Uha$-algebras.
\begin{proof}
Let $f: \U \hhh \to A$ be a $\hlie$-trivialization of $A$. We denote by $V(n)$ ($n \in \ZZ$) the pullback of the representation $\Verma \hhh$ by $f^{-1}$. We denote by $U(\pm 1)$ the subspace of $\U$ formed by the elements $x$ such that $[H,x] = \pm 2 x$. As $f(H) = H$, and in view of the relation $[H,X^\pm] = \pm 2 X^\pm$ in $A$, $f^{-1}(X^\pm)$ belongs to $U(\pm 1) \hhh$. This implies that in $V(n)$, $X^+.b_0 = 0$, $X^-.b_k = \alpha^k(n) b_{k+1}$ and $X^+.b_{k+1} = \beta^k(n) b_k$ for some $\alpha^k(n), \beta^k(n) \in \PSR$ ($k \geq 0$). Let us denote by $b'_k$ ($k \geq 0$) the vector $(X^-)^k.b_0$ in $V(n)$. As $\hzero f$ is by assumption an $\alie$-algebra homomorphism, and in view of definition \eqref{eq:Preliminaries/Colourings/Verma} of $\Verma$, $\alpha^k(n)$ and $\beta^k(n)$ are equal to $1$ and to $(k+1)(n-k)$, respectively, modulo $h$. This proves first that $b'_k$ is a non-zero scalar multiple of $b_k$ for all $k \geq 0$, and then that $V(n) = (\bigoplus_{k \geq 0} \field b'_k) \hhh$ as a $\PSR$-module, with
\begin{align*}
	H . b_k \,& = \, (n-2k) b_k \quad \text{(since $f(H) = H$)},\\
	X^- . b'_k \,& = \, b'_{k+1},\\
	X^+ . b'_k \,& = \,
	\begin{cases}
		0& \text{if $k = 0$,}\\
		\psi^k(n) b'_{k-1}& \text{if $k \geq 1$,}
	\end{cases}
\end{align*}
for $\psi^k(n) \in \PSR$ such that $\psi^k(n) = k(n-k+1)$ modulo $h$. If $n \geq 0$, then $\bigoplus_{k \geq n+1} \field b_k$ is a subrepresentation of $\Verma$, therefore $(\bigoplus_{k \geq n+1} \field b'_k) \hhh$ is a subrepresentation of $V(n)$. This implies that $\psi^{n+1}(n)$ has to be zero for every $n \geq 0$. It then follows from step 1 that $\psi^{n+k+1}(n) = \psi^k(-n-2)$ for all $k \geq 1$ and all $n \geq 0$. We thus have proved that the values $\psi^k(n)$ ($k \geq 1$, $n \in \ZZ$) define a colouring $\psi$, such that for all $n \in \ZZ$ the action of $\Uha$ on $\Vermah$ factorises through $A$. It follows from step 2 that $A$ and $\Uh$ are isomorphic as $\Uha$-algebras. In particular, $\Uh$ is a formal deformation of the $\alie$-algebra $\U$. Hence, the colouring $\psi$ is regular (theorem \ref{thm:CKM/Formal-deformations}).
\end{proof}

\subsubsection*{Conclusion.} Step 3 proves that every $\hlie$-trivial formal deformation of $\U$ is isomorphic as a $\Uha$-algebra to $\Uh$, for some regular colouring $\psi$. It remains to prove that for two regular colourings $\psi$ and $\psi'$, if $\Uh[\psi]$ and $\Uh[\psi']$ are isomorphic as $\Uha$-algebras, then $\psi = \psi'$. Let $n \in \ZZ$. By definition, the actions of $\Uha$ on the representations $\Vermah$ and $\Vermah[n,\psi']$ factorise through $\Uh$ and $\Uh[\psi']$, respectively. Let us suppose that $\Uh[\psi]$ and $\Uh[\psi']$ are isomorphic as $\Uha$-algebras. Then, the action of $\Uha$ on $\Vermah[n,\psi']$ also factorises through $\Uh$. As $\Uh$ is a formal deformation of the $\alie$-algebra $\U$ (theorem \ref{thm:CKM/Formal-deformations}), it follows from step 1 that $\psi^k(n) = (\psi')^k(n)$ for all $k \geq 1$.
\end{proof}

% ---------- Body of the document ---------- %

% ---------- Bibliography ---------- %

% ---------- Bibliography ---------- %

\affiliationone{
Alexandre Bouayad \\
Department of Pure Mathematics and Mathematical Statistics \\
University of Cambridge \\
Cambridge CB3 0WB \\
United Kingdom
\email{A.Bouayad@dpmms.cam.ac.uk}}
\end{document}